\documentclass[a4paper,12pt, reqno]{amsart}
\usepackage{amssymb,amsthm,amsmath}
\usepackage{cite}
\usepackage{amsmath,amsthm,amscd,amssymb}
\usepackage{latexsym}
\usepackage[colorlinks,citecolor=blue,pagebackref,hypertexnames=false]{hyperref}
\numberwithin{equation}{section}
\theoremstyle{plain}
\newtheorem{theorem}{Theorem}[section]

\newtheorem{lemma}[theorem]{Lemma}
\newtheorem{corollary}[theorem]{Corollary}
\newtheorem{proposition}[theorem]{Proposition}

\theoremstyle{definition}
\newtheorem{definition}[theorem]{Definition}

\newtheorem{example}[theorem]{Example}

\theoremstyle{remark}
\newtheorem{remark}[theorem]{Remark}

\newtheorem{case[theorem]}{Case}

\def \R{{\mathbb R}}

\def\norm#1.#2.{\lVert#1\rVert_{#2}}

\def\R{\mathbb R}

\title[$M$-elliptic pseudo-differential operators  on $\mathbb{Z}^n$]
{ Symbolic calculus and  $M$-ellipticity of pseudo-differential operators on  $\mathbb{Z}^n$ 
}

\author{Vishvesh Kumar}\address{Vishvesh Kumar  \endgraf Department of Mathematics: Analysis, Logic and Discrete Mathematics	\endgraf Ghent University	\endgraf Krijgslaan 281, Building S8,	B 9000 Ghent,	Belgium.} \email{vishveshmishra@gmail.com}

\author{Shyam Swarup Mondal } \address{Shyam Swarup Mondal  \endgraf Department of Mathematics	\endgraf Indian Institute of Technology Guwahati,	\endgraf Guwahati-781039, Assam, India.} \email{mondalshyam055@gmail.com}

\date{\today}
\begin{document}

	\maketitle

	\allowdisplaybreaks
	
	\begin{abstract}
		In this paper, we introduce and study a class of pseudo-differential operators on the lattice  $\mathbb{Z}^n$. More preciously, we consider a weighted symbol class $M_{\rho, \Lambda}^m(\mathbb{ Z}^n\times \mathbb{T}^n), m\in \mathbb{R}$ associated to a suitable weight function  $\Lambda$ on $\mathbb{ Z }^n$.  We study elements of the symbolic calculus for pseudo-differential operators associated with $M_{\rho, \Lambda}^m(\mathbb{ Z}^n\times \mathbb{T}^n)$  by deriving formulae for the composition, adjoint, transpose. We define the notion of $M$-ellipticity for symbols belonging to $M_{\rho, \Lambda}^m(\mathbb{ Z}^n\times \mathbb{T}^n)$ and construct the parametrix of $M$-elliptic  pseudo-differential operators.  Further, we investigate the minimal and maximal extensions for $M$-elliptic pseudo-differential operators and show that they coincide on $\ell^2(\mathbb{Z}^n)$ subject to the  $M$-ellipticity of symbols. We also determine the domains of the minimal and maximal operators. Finally, We discuss Fredholmness and compute the index of  $M$-elliptic pseudo-differential operators on $\mathbb{Z}^n$.% Finally, we establish the Gårding   inequality for $M$-elliptic pseudo-differential operators with an application to the unique solvability of parabolic equations on the lattice $\mathbb{ Z}^n$. 
	\end{abstract}

	\tableofcontents 	
	
	\section{Introduction}
	\noindent  This paper contributes to the theory of pseudo-differential operators on $\mathbb{Z}^n$ by developing a symbolic calculus for  pseudo-differential operators associated with  weighted symbol class $M_{\rho, \Lambda}^m(\mathbb{ Z}^n\times \mathbb{T}^n)$. We also discuss results concerning to the minimal and maximal extensions  of operators corresponding to this symbol class and study Fredholmness and $M$-ellipticity of discrete
	pseudo-differential operators  on $\mathbb{Z}^n$.  	The pseudo-differential operators on $\mathbb{Z}^n$ are suitable for solving difference equations on $\mathbb{Z}^n$ (see \cite{bot,DTR}). Difference equations very often appear in mathematical modeling problems and in the discretization of continuous problems.  There are various physical models realised as difference equations, see e.g.\cite{rabin06,rabin13} for the analysis of Schr\"odinger, Dirac, and many other such operators on $\mathbb{Z}^n$ and  their spectral properties.   A suitable theory of pseudo-differential operators on $\mathbb{ Z}$ and on the lattice  $\mathbb{ Z}^n$  has been developed by several authors, but without symbolic calculus \cite{Delgado, Shala3,rabi10,rab2,Kumar202020}. Very recently, a global symbolic calculus for psuedo-differential operators on the lattice $\mathbb{ Z}^n$ has been constructed by  Botchway, Kabiti and Ruzhansky \cite{bot}. They called pseudo-differential operators on $\mathbb{Z}^n$ as {\it pseudo-difference operators} and presented several applications (see also \cite{DTR}). An interesting and important   contrast  between the standard theory of pseudo-differential operators  on $\mathbb{R}^n $ and the theory of pseudo-difference operators on $\mathbb{Z}^{n}$ is that, since $\mathbb{Z}^{n}$ is discrete, the Schwartz kernels of the corresponding pseudo-differencial operators on $\mathbb{Z}^n$ do not have singularities at the diagonal. Also, it is worth noting that 
	the  phase space in this case  is $\mathbb{Z}^{n} \times \mathbb{T}^{n}$ with the frequencies being elements of the compact space $\mathbb{T}^{n}$.    The classical theory of pseudo-differential operators depends on symbol classes with increasing decay of symbols after taking their derivatives with respect to the frequency variable. However,  in the case of $\mathbb{Z}^n$,  since the frequency space $\mathbb{T}^{n}$ is compact,   one can not construct the discrete calculus using standard methods relying on the decay properties with respect to  the frequency component of the phase space and   any  improvement with respect to the decay of the frequency variable is not possible.  Thus the discrete symbolic calculus is essentially different  from usual symbolic calculus on $\mathbb{R}^n$.   Since the usual derivatives are not suitable operators while dealing with the space variable $k \in \mathbb{Z}^{n}$, it is convenient to  work with appropriate difference operators on $\mathbb{Z}^n$.   The global theory of discrete pseudo-differential operators and corresponding  symbolic calculus have been developed in \cite{bot} and one can look \cite{ruz1, Ruz14,turrrr,  Ruz11, RuzT, turr, turrr} for  the similar  theory but   swapping the order of the space and frequency variables. Some recent developments on the theory of pseudo-differential operators on $\mathbb{ Z}^n$ and its applications can be found in  \cite{CardonaKumar,CardonaKumar1,Cardonaaa} and references therein.   For  theory of pseudo-differential operators on other classes of groups, we refer to    \cite{bogg, taylor2020, CardonaKumar, DK, Vish,  RuzT,RT13,fis}.\\

	%	One of the most important questions in the theory of pseudo-differential operators is that under which conditions the minimal and maximal extensions of pseudo-differential operators in 	$L^p (\mathbb{R}^n), 1 < p < \infty,$ will coincide. Wong in \cite{wong99} answered the above question in terms of elliptic operator and proved that the domain of these operators coincides with the Sobolev space $H^{m, p}\left(\mathbb{R}^{n}\right)$.

	\noindent   It is well known that pseudo-differential operators are $L^{p}$-bounded  when $p=2$ for the class $S_{\rho, \delta}^{0}$ with $0 \leq \delta<\rho \leq 1$ (see \cite{two,six}) but Fefferman in \cite{feff} proved that  the pseudo-differential operators whose symbols belong to the class $S_{\rho, 0}^{0}$  with $0<\rho<1$ are not $L^{p}$-bounded  in general  for  $p \neq 2$. To avoid the  difficulty described above, 
	Taylor in \cite{taylor81} introduced a suitable symbol subclass $M_{\rho}^{m}$ of $S_{\rho, 0}^{0}$ and  developed   symbolic  calculus for the associated pseudo-differential operators.  Later, Garello and Morando \cite{gar,mor05} introduced the subclass  $M_{\rho, \Lambda}^{m}$  of $S_{\rho, \Lambda}^{m}$, which are just a weighted version of the symbol class introduced by Taylor and developed the symbolic calculus for the associated pseudo-differential operators  with   many applications to  study the  regularity of multi-quasi-elliptic     operators.  	Here,   we  % consider a weighted symbol class $M_{\rho, \Lambda}^m(\mathbb{ Z}^n\times \mathbb{T}^n), m\in \mathbb{R}$ associated to a suitable weight function  $\Lambda$ on $\mathbb{ Z }^n$ and we
	introduce a weighted symbol class $M_{\rho, \Lambda}^m(\mathbb{ Z}^n\times \mathbb{T}^n), m\in \mathbb{R}$ associated to a suitable weight function  $\Lambda$ on $\mathbb{ Z }^n$  and study elements of the symbolic calculus for  pseudo-differential operators  associated  with symbol in $M_{\rho, \Lambda}^m(\mathbb{ Z}^n\times \mathbb{T}^n)$. Note that $M_{\rho, \Lambda}^m(\mathbb{ Z}^n\times \mathbb{T}^n)$  is a subclass of Taylor's symbol class   $S_{\rho, \Lambda}^{m} \left( \mathbb{Z}^n \times \mathbb{T}^n \right) $ on $\mathbb{Z}$  which is  just a weighted version of the    H\"ormander symbol class $ S_{\rho}^{m}\left(\mathbb{Z}^{n} \times \mathbb{T}^{n}\right)$ on $\mathbb{Z}^n$,     introduced by Botchway, Kabiti and Ruzhansky\cite{bot}.    Considerable attention has been devoted  to study   various properties of pseudo-differential operators associated with the symbol class $M_{\rho, \Lambda}^{m}$ on $\mathbb{R}^n$ by several researchers. In particular, Wong \cite{wong06} studied  asymptotic expansions and parametrix for the  pseudo-differential operators on $L^p(\mathbb{R}^n)$ whose symbol is in    $M_{\rho, \Lambda}^{m}, m>0$. One of the important questions arising in the study of pseudo-differential operators  is to determine the
	conditions under which the “minimal” and “maximal” extensions of a pseudo-differential operators on
	$L^p(\mathbb{R}^n), 1 < p < \infty$, coincide.  Wong\cite{wong06} investigated this question and showed that that  the minimal and maximal extensions on $L^p(\mathbb{R}^n)$ for $M$-elliptic pseudo-differential operators coincides and determined the domains of the minimal and maximal operators. The minimal and maximal extensions of $M$-elliptic pseudo-differential operators  on $L^p(\mathbb{T}^n)$ were studied by Alimohammady and Kalleji  in \cite{alim13}.	Later,  a symbolic calculus, minimal and maximal extensions of $M$-elliptic pseudo-differential operators  in  $L^p(\mathbb{T}^n)$,  among some other things, studied   by Kalleji in \cite{kal}.    One can look \cite{publo,mor16} for the minimal and maximal extension pseudo-differential operators in different settings. Recently, Dasgupta and Kumar \cite{vish}  discussed the ellipticity and Fredholmness of pseudo-differential operators on $\ell^2(\mathbb{ Z}^n)$ in the context of minimal and maximal operators. In particular, the authors proved that the minimal and maximal extensions of pseudo-differential operators coincide when the elliptic symbol lies in the class $S^m(\mathbb{ Z}^n\times \mathbb{T}^n)$, the symbol class constructed by  Botchway, Kabiti and Ruzhansky  \cite{bot}. They also   discussed  Fredholmness property of discrete pseudo-differential operators. The Fredholmness measures the almost invertibility of the operators and depends  on the space where the operators
	are defined.  The relation between   ellipticity and Fredholmness property is crucial in the  theory of pseudo-differential operators on a given function space.    The equivalence of ellipticity and Fredholmness and the applications of SG pseudo-differential operators with positive order on $L^p(\mathbb{R}^n), 1 < p < \infty$, has been investigated in  \cite{apara, wong99}. %Also, the authors in    \cite{garding6} studied  $M$-Ellipticity of Fredholm Pseudo-Differential Operators on $L^p(\mathbb{R}^n)$ and G\r{a}rding's Inequality.	
	However,  such questions for pseudo-differential operators  on $\mathbb{Z}^n$ associated with weight symbol class on $\mathbb{Z}^n$ have not been studied yet. In this paper, we introduce the notion of $M$-ellipticity for psuedo-differential operators on $\mathbb{Z}^n$ and construct the minimal and maximal extensions for $M$-elliptic pseudo-differential operators. We  will show that  they coincide on $\ell^2(\mathbb{Z}^n)$ subject to the  $M$-ellipticity of symbols. We also study Fredholmness property of discrete pseudo-differential operators and   compute the index of $M$-elliptic pseudo-differential 	operators on $\mathbb{Z}^n$.  % Finally, we establish Gårding   inequality for $M$-elliptic pseudo-differential operators with an application to the unique solvability of parabolic type equations on the lattice $\mathbb{ Z}^n$.   Note that  the  Gårding inequality %on $\mathbb{R}^n$ 	 is one of the most important tools in field of  the microlocal analysis with numerous applications in the theory of partial differential equations. There is a vast literature on this subject which is impossible to cite here; we refer  \cite{garding1,garding2,garding3,garding5,Ruz11,Ruz14,garding6}   to mention a very few recent important works.
	\\

	\noindent  The presentation of this manuscript is divided into six sections including the introduction.   In section  \ref{sec2}, we recall basics of Fourier analysis and pseudo-differential operators  on $\mathbb{Z}^n.$ In Section \ref{secwe}, we introduce  the weighted symbol class $M_{\rho, \Lambda}^{m}(\mathbb{Z}^n \times \mathbb{T}^n ), m\in \R$  associated to a suitable weight function  $\Lambda$ on $\mathbb{ Z }^n$ and study several important properties of weighted symbol class $M_{\rho, \Lambda}^{m}(\mathbb{Z}^n \times \mathbb{T}^n )$. 
	Elements of   symbolic calculus for  pseudo-differential operators associated with $M_{\rho, \Lambda}^m(\mathbb{ Z}^n\times \mathbb{T}^n)$  by deriving formulae for the composition, adjoint, transpose is given  in   Section \ref{sec3}.  We also define the nation of $M$-ellipticity for symbols belonging to $M_{\rho, \Lambda}^m(\mathbb{ Z}^n\times \mathbb{T}^n)$ and construct the parametrix of $M$-elliptic  pseudo-differential operators.  In section \ref{sec4},  we construct the minimal and maximal extensions of $M$-elliptic pseudo-differential operators  and show that they coincide in $\ell^2(\mathbb{Z}^n)$  under $M$-ellipticity assumption and determine the domains of the minimal, and hence maximal operators.  Finally, in Section  \ref{sec6}, we discuss Fredholmness  and   compute the index of   $M$-elliptic pseudo-differential operators on $\mathbb{Z}^n$. % We derive the Gårding’s inequality for  $M$-elliptic pseudo differential opertors in 	In Section \ref{sec7}. Finally, in Section \ref{sec8}, we show unique solvability of parabolic equations on the lattice $\mathbb{ Z}^n$ as an application of  Gårding   inequality for $M$-elliptic operators.
	\\

	\section{Fourier analysis and pseudo-differential operators on $\mathbb{Z}^n$} \label{sec2}
	In this section,  we first recall some notation and basic properties of discrete Fourier
	analysis and pseudo-differential operators on $\mathbb{Z}^n$. We refer \cite{bot,RuzT,CardonaKumar1,vish,Delgado,rab2,rabi10} for more details and the study of various operator theoretical and mapping properties of these objects.
	The  Fourier transform $\hat{f}$ of a function $f \in \ell^{1}\left(\mathbb{Z}^{n}\right)$ is defined by
	$$
	\widehat{f}(x)=\sum_{k \in \mathbb{Z}^{n}} e^{-2 \pi i k \cdot x} f(k)
	$$
	for all $x \in \mathbb{T}^{n}$. The Fourier transform on $\mathbb{Z}^n$ is also known as discrete Fourier transform. The discrete Fourier transform can be extended to $\ell^{2}\left(\mathbb{Z}^{n}\right)$ using
	the standard density arguments. We normalize the Haar measures on $\mathbb{Z}^{n}$ and $\mathbb{T}^{n}$ in such a
	manner so that the following Plancherel formula  holds:
	$$
	\sum_{k \in \mathbb{Z}^{n}}|f(k)|^{2}=\int_{\mathbb{T}^{n}}|\widehat{f}(x)|^{2} d x.
	$$
	The inverse of discrete Fourier transform is given by
	$$
	f(k)=\int_{\mathbb{T}^{n}} e^{2 \pi i k \cdot x} \widehat{f}(x) d x, \quad k \in \mathbb{Z}^{n},
	$$ where $f$ belongs to a suitable function space.
	Let $f$ be a function on $\mathbb{Z}^{n}$ and $e_{j} \in \mathbb{N}^{n}$
	be such that $e_{j}$ has 1 in the $j$-th entry and zeros elsewhere. The  difference operator $\Delta_{k_{j}}$
	is defined by
	$$
	\Delta_{k_{j}} f(k)=f\left(k+e_{j}\right)-f(k)
	$$
	and set
	$$
	\Delta_{k}^{\alpha}=\Delta_{k_{1}}^{\alpha_{1}} \Delta_{k_{2}}^{\alpha_{2}} \ldots \Delta_{k_{n}}^{\alpha_{n}}
	$$
	for all
	$
	\alpha=\left(\alpha_{1}, \alpha_{2}, \ldots, \alpha_{n}\right) \in \mathbb{N}_{0}^{n}=\mathbb{N}^{n} \cup\{0\}
	$
	and  for all 
	$
	k=\left(k_{1}, k_{2}, \ldots, k_{n}\right) \in \mathbb{Z}^{n}.
	$
	Then we have the following well known formulae  from \cite{bot}, which will be  useful in the sequel.
	\begin{enumerate}
		\item[(i)]  \label{s7}
		For a complex-valued function $f$ on $\mathbb{Z}^n,$ we have
		$$
		\left(\Delta_{k}^{\alpha} f \right)(k)=\sum_{\beta \leq \alpha}(-1)^{|\alpha-\beta|}\left(\begin{array}{c}
			\alpha \\
			\beta
		\end{array}\right) f(k+\beta), \quad k \in \mathbb{Z}^{n}.
		$$
		\item[(ii)]  \label{s8}
		For $f, g: \mathbb{Z}^{n} \rightarrow \mathbb{C},$ the following Leibniz formula holds
		$$
		\left(\Delta_{k}^{\alpha}(f g)\right)(k)=\sum_{\beta \leq \alpha}\left(\begin{array}{l}
			\alpha \\
			\beta
		\end{array}\right)\left(\Delta_{k}^{\beta} f\right)(k)\left({\Delta}_{k}^{\alpha-\beta} g\right)(k+\beta), \quad k \in \mathbb{Z}^{n}
		$$
		for all multi-indices $\alpha .$ 
		\item[(iii)]  \label{s9} 
		For two functions $f, g: \mathbb{Z}^{n} \rightarrow \mathbb{C}$, the following identity is true
		$$
		\sum_{k \in \mathbb{Z}^{n}} f(k)\left(\Delta_{k}^{\alpha} g\right)(k)=(-1)^{|\alpha|} \sum_{k \in \mathbb{Z}^{n}}\left(\bar{\Delta}_{k}^{\alpha} f\right)(k) g(k), \quad k \in \mathbb{Z}^{n}
		$$
		where $\left(\bar{\Delta}_{k_{j}} f\right)(k)=f(k)-f\left(k-e_{j}\right).$
	\end{enumerate}
	\begin{remark}\label{s10}
		For $k \in \mathbb{Z}^{n}$ and $\alpha \in \mathbb{N}_{0}^{n},$ we define $k^{(\alpha)}=k_{1}^{\left(\alpha_{1}\right)} \ldots k_{n}^{\left(\alpha_{n}\right)},$ where $k_{j}^{(0)}=1$
		and $k_{j}^{(m+1)}=k_{j}^{(m)}\left(k_{j}-m\right)=k_{j}\left(k_{j}-1\right) \ldots\left(k_{j}-m\right) .$ Then $$\Delta_{k}^{\gamma} k^{(\alpha)}=\alpha^{(\gamma)} k^{\alpha-\gamma}.$$
	\end{remark}
	% We define $\Delta^{\alpha}$ acting on functions $\tau: \mathbb{Z}^{n} \rightarrow \mathbb{C}$ by the formula$$\Delta^{\alpha} \tau(k):=\int_{\mathbb{T}^{n}} e^{2 \pi i k \cdot y}\left(e^{2 \pi i y}-1\right)^{\alpha} \widehat{\tau}(y) \mathrm{d} y $$ where $\alpha=\left(\alpha_{1}, \ldots, \alpha_{n}\right)$ and $$ \left(e^{2 \pi i y}-1\right)^{\alpha}=\left(e^{2 \pi i y_{1}}-1\right)^{\alpha_{1}} \cdots\left(e^{2 \pi i y_{n}}-1\right)^{\alpha_{n}} $$Notice that  $$ \Delta^{\alpha}=\Delta_{1}^{\alpha_{1}} \cdot \ldots \cdot \Delta_{n}^{\alpha_{n}}. $$ Denoting $v_{j}=(0, \cdots, 0,1,0, \cdots, 0)$ with 1 at the $j$ th position, we have$$ \begin{aligned} \Delta_{j} \tau(k) &=\int_{\mathbb{T}^{n}} e^{2 \pi i\left(k+v_{j}\right) \cdot y} \widehat{\tau}(y) \mathrm{d} y-\int_{\mathbb{T}^{n}} e^{2 \pi i k \cdot y} \widehat{\tau}(y) \mathrm{d} y \\ &=\tau\left(k+v_{j}\right)-\tau(k) \end{aligned} $$ are the usual difference operators on $\mathbb{Z}^{n}$.  The formula ( 2.1 ) makes sense for $\tau \in \mathcal{S}^{\prime}\left(\mathbb{Z}^{n}\right) .$ Indeed, in this case we have $\widehat{\tau} \in \mathcal{D}^{\prime}\left(\mathbb{T}^{n}\right)$ and the formula (2.1) can be interpreted in terms of the distributional duality on $\mathbb{T}^{n}$, $$ \Delta^{\alpha} \tau(k)=\left\langle\widehat{\tau}, e^{2 \pi i k \cdot y}\left(e^{2 \pi i y}-1\right)^{\alpha}\right\rangle $$ acting on the $y$ -variable. These operators have been introduced, analysed and shown to satisfy many useful properties, such as the Leibniz formula, summation by parts formula,
	We will  use the following usual nomenclature,
	$$ D_{x}^{\alpha}=D_{x_{1}}^{\alpha_{1}} \cdots D_{x_{n}}^{\alpha_{n}}, \quad D_{x_{j}}=\frac{1}{2 \pi i} \frac{\partial}{\partial x_{j}} $$ and  $$ D_{x}^{(\alpha)}=D_{x_{1}}^{\left(\alpha_{1}\right)} \cdots D_{x_{n}}^{\left(\alpha_{n}\right)}, \quad D_{x_{j}}^{(\ell)}=\prod_{m=0}^{\ell}\left(\frac{1}{2 \pi i} \frac{\partial}{\partial x_{j}}-m\right), \quad \ell \in \mathbb{N}.$$
	As usual, $D_{x}^{0}=D_{x}^{(0)}=I .$ The operators $D_{x}^{(\alpha)}$ become very useful in the analysis related to the torus as they appear in the periodic Taylor expansion \cite{ruz1}. 
	
	Let us now recall the H\"ormander symbol classes $ S_{\rho}^{m}\left(\mathbb{Z}^{n} \times \mathbb{T}^{n}\right)$ on $\mathbb{Z}^n$ defined in \cite{bot}. Note that the following symbol classes were denoted by $S^m_{\rho, 0}(\mathbb{Z}^n \times \mathbb{T}^n)$ in \cite{bot}. 
	\begin{definition} \label{eq10} Let $m \in \mathbb{R}$ and $\rho>0.$ We say that a function $\sigma: \mathbb{Z}^{n} \times \mathbb{T}^{n} \rightarrow \mathbb{C}$ belongs to $ S_{\rho}^{m}\left(\mathbb{Z}^{n} \times \mathbb{T}^{n}\right)$  if $\sigma(k, \cdot) \in C^{\infty}\left(\mathbb{T}^{n}\right)$ for all $k \in \mathbb{Z}^{n},$ and for
		all multi-indices $\alpha, \beta$, there exists a positive constant $C_{\alpha, \beta}$ such that  
		$$
		\left|D_{x}^{(\beta)} \Delta_{k}^{\alpha} \sigma(k, x)\right| \leq C_{\alpha, \beta}(1+|k|)^{m-\rho|\alpha|}, \quad (k, x)\in  \mathbb{Z}^{n} \times \mathbb{T}^{n}.
		$$
		
	\end{definition}
	%We  simply denote $ S_\rho^{m}:=S_{\rho}^{m}\left(\mathbb{Z}^{n} \times \mathbb{T}^{n}\right)$.
	For the symbol $\sigma\in S_{\rho}^{m}\left(\mathbb{Z}^{n} \times \mathbb{T}^{n}\right)$, 	the corresponding pseudo-differential  operator associated with $\sigma$ is given by
	$$
	(T_\sigma f) (k):=\int_{\mathbb{T}^{n}} e^{2 \pi i k \cdot x} \sigma(k, x) \hat{f}(x) \mathrm{d} x, \quad k \in \mathbb{Z}^{n}
	$$
	for every $f \in C^{\infty}(\mathbb{Z}^n)$. 	We denote  $\mathrm{OP} S_\rho^{m}\left(\mathbb{Z}^{n} \times \mathbb{T}^{n}\right)$ be the set of all operators  corresponding to the symbol class $S_\rho^{m}\left(\mathbb{Z}^{n} \times \mathbb{T}^{n}\right)$. 

	\section{Weighted symbol class $M_{\rho, \Lambda}^{m}(\mathbb{Z}^{n} \times \mathbb{T}^{n})$ and its properties} \label{secwe}
	
	This section is devoted to introduce the weighted symbol classes  $M_{\rho, \Lambda}^m(\mathbb{Z}^n \times \mathbb{T}^n )$ in conjunction with a suitable weight function $\Lambda $ on $\mathbb{ Z}^n$. These symbol classes in the setting of Euclidean spaces were introduced in \cite{gar, mor05} and which can be traced back to the excellent book of Taylor \cite{taylor81}.
	We start with the definition of weight function of $\mathbb{Z}^n$.
	
	\begin{definition}\label{definition} Let $ \Lambda \in C^{ \infty }\left(\mathbb{Z}^n \right)$. We say that  $\Lambda$ is a weight function if there exist
		suitable positive constants $\mu_{0} \leq \mu_{1} \leq \mu$ and $C_{0}, C_{1}$ such that
		\begin{equation} \label{growth}
			C_{0}( 1+  |k|  )^{\mu_{0} } \leq \Lambda( k) \leq C_{1} (1+  |k| )^{\mu_{1} }
		\end{equation}  and for every multi-indices $ \alpha \in \mathbb{N}_0^n,  \gamma \in \{0, 1\}^n$  and  some positive constant $C_{\alpha, \gamma}$ such that 
		$$\left| k^{\gamma} \Delta_k^{\alpha+\gamma } \Lambda(k)\right| \leq C_{\alpha, \gamma} \Lambda(k)^{1-(1 / \mu)|\alpha|}$$
		for all $k \in \mathbb{Z}^n,$ where agreeing the standard multi-index notation $k^{\gamma} =k_1^{\gamma_1} \cdots k_n^{\gamma_n}  .$
		
	\end{definition}
	\begin{example}\quad 
		\begin{itemize}
			\item[(i)] For a positive  real number $m$, define $\Lambda_m$ on $\mathbb{Z}^n$ by   $$\Lambda_m(k)=\sqrt{1+k_1^{2m}+\cdots + k_n^{2m}}.$$ Then $\Lambda_m$ is a example of  weight function on $\mathbb{Z}^n$ of order $m$.
			\item[(ii)]  Similarly, the function $\Lambda_{m_1, \ldots, m_n}(k)=\sqrt{1+k_1^{2m_1}+\cdots + k_n^{2m_n}}$  is also a  weight function on $\mathbb{Z}^n$, where $(m_1, \ldots, m_n)\in \mathbb{N}^n$ with $\inf_j m_j\geq 1.$
		\end{itemize}
	\end{example}
	Now, we define the following Taylor symbol class $S_{\rho, \Lambda}^{m} \left( \mathbb{Z}^n \times \mathbb{T}^n \right) $ associated to a suitable weight function $\Lambda$ on $\mathbb{Z}^n.$ This class of symbols can be thought as the weighted H\"ormander class on $\mathbb{Z}^n.$
	
	\begin{definition}
		Let $m \in \mathbb{R} $ and $\rho\in (0, \frac{1}{\mu}]$. We define Taylor symbol class   $S_{\rho, \Lambda}^{m} \left( \mathbb{Z}^n \times \mathbb{T}^n \right) $ as the set of all symbols $\sigma: \mathbb{Z}^n \times \mathbb{T}^n \rightarrow \mathbb{C}$ which are smooth in $x$ for all $k \in \mathbb{Z}^n$  and for all  multi-indices $\alpha, \beta \in \mathbb{Z}^n_{+} $, there exists a constant $C_{\alpha, \beta, m}>0$ such that 
		\begin{align}\label{s1}
			\left| D_{x}^{(\beta)} \Delta_{k}^{\alpha}\sigma(k, x)\right| \leq C_{ \alpha, \beta, m}\Lambda(k)^{m-\rho|\alpha|}, \quad (k, x)\in  \mathbb{Z}^{n} \times \mathbb{T}^{n}.
		\end{align}
	\end{definition}
	\begin{remark}
		When $\Lambda(k)=\Lambda_1(k)=(1+|k|^2)^{\frac{1}{2}},  k\in \mathbb{Z}^n$, then the symbol class $S_{ \rho, \Lambda}^{m}\left( \mathbb{Z}^n \times \mathbb{T}^n \right) $ coincides with the H\"ormander class $ S_\rho^{m}\left( \mathbb{Z}^n \times \mathbb{T}^n \right) $ for $m \in \mathbb{R}$ as defined in  Definition \ref{eq10}.
	\end{remark} 
	%	\begin{remark}\label{newremark}		From Remark \ref{s10} and  the relation (3.14) of \cite{kal}, we have 	$$ 	\left|\Delta_{k}^{\eta}k^{\gamma}\right| \leq C_{\alpha, \gamma} \Lambda(k)^{\rho|\gamma-\eta|}, \quad k \in \mathbb{Z}^n	$$ 	for every multi-indices $ \eta \in \mathbb{N}_0^n$ and $  \gamma \in \{0, 1\}^n$. 	\end{remark} 
	
	As usual we set $$S_{\rho, \Lambda}^{\infty}\left( \mathbb{Z}^n \times \mathbb{T}^n \right):=\bigcup_{m \in \mathbb{R}} S_{\rho, \Lambda}^{m}\left( \mathbb{Z}^n \times \mathbb{T}^n \right) $$ and $$ S_{\rho, \Lambda}^{-\infty}\left( \mathbb{Z}^n \times \mathbb{T}^n \right):=\bigcap_{m \in \mathbb{R}} S_{\rho, \Lambda}^{m}\left( \mathbb{Z}^n \times \mathbb{T}^n \right).$$ 
	Define the pseudo-differential operator $T_\sigma$ associated with symbol $\sigma\in  S_{ \rho, \Lambda}^{m}\left( \mathbb{Z}^n \times \mathbb{T}^n \right)$ as
	\begin{align}\label{s12}
		T_\sigma f(k)=\int_{\mathbb{T}^n}e^{2\pi i k\cdot x} \sigma (k, x) \widehat{f}(x)  ~dx, ~k\in  \mathbb{Z}^n
	\end{align}
	for every $f \in C^{\infty}(\mathbb{Z}^n)$. We  write $\operatorname{OP}S_{\rho, \Lambda}^{m}\left( \mathbb{Z}^{n} \times \mathbb{T}^{n} \right)$ for the class of  pseudo-differential operators associated with the symbol class  $ S_{\rho, \Lambda}^{m}\left( \mathbb{Z}^{n} \times \mathbb{T}^{n} \right)$.
	
	Now, we are ready to define the main ingredient of this paper, namely the symbol class  $M_{\rho, \Lambda}^{m}\left( \mathbb{Z}^n \times \mathbb{T}^n \right).$
	\begin{definition}\label{0}
		For $m \in \mathbb{R}$ and $\rho\in (0, \frac{1}{\mu}]$, we define the  symbol class   $M_{\rho, \Lambda}^{m}\left( \mathbb{Z}^n \times \mathbb{T}^n \right) $ to be the class of all  $\sigma:   \mathbb{Z}^n \times \mathbb{T}^n\to \mathbb{C}$ which are smooth in $x$ for all $k \in \mathbb{Z}^n$ such that for all  $\gamma \in\{0,1\}^n$
		\begin{align}\label{s2}
			k^\gamma \Delta_k^{\gamma} \sigma(k, x)\in S_{ \rho, \Lambda}^{m}\left( \mathbb{Z}^n \times \mathbb{T}^n \right) .
		\end{align}
	\end{definition}
	
	%	\begin{remark}	With the assumptions in Definitions  \ref{0},  we notice that the class $M_{ \rho, \Lambda}^{m}$ also can be characterized by the following estimate: for every compact set $K\subset  \mathbb{Z}^n$		\begin{align}\label{s3}		\left| k^\gamma D_{x}^{(\beta)} \Delta_{k}^{\alpha+\gamma} \sigma(k, x)\right| \leq C_{\sigma, \alpha, \beta, m}\Lambda(k)^{m-\rho|\alpha|}	\end{align}	for all  $ k \in K$, $ x \in \mathbb{T}^n, \gamma \in\{0,1\}^n$, and   $\alpha, \beta \in \mathbb{Z}_{+}^{n}$.	\end{remark}

	%	\begin{remark} \cite{ruz1}
		%		In general on $\mathbb{Z}^{n}$, Hörmander's usual $(\rho, \delta)$-class of pseudo-differential operators $\operatorname{OP} S_{\rho, \Lambda}^{m}\left(\mathbb{R}^{n} \times \mathbb{R}^{n}\right)$ of order $m \in \mathbb{R}$ which are $2 \pi$-periodic in dual variable  does not coincides with the class $\operatorname{OP} S_{\rho, \Lambda}^{m}\left( \mathbb{Z}^{n}\times \mathbb{T}^{n}  \right) $ (see Remark  2.7, \cite{kal}). But in special case when $\delta=0$ and $\Lambda(k)= (1+|k|^2)^{\frac{1}{2}} $, then  these two classes of symbols   coincides, i.e.,
		%		$$
		%		\operatorname{O P}S_{\rho, \Lambda}^{m}\left(\mathbb{R}^{n} \times \mathbb{R}^{n}\right)=\operatorname{O P} S_{\rho, \Lambda}^{m}\left( \mathbb{Z}^{n} \times \mathbb{T}^{n} \right).
		%		$$
		%	\end{remark}
	
	\begin{remark}\label{eq6}
		For every $m \in \mathbb{R}$ and $\rho\in (0, \frac{1}{\mu}]$,  we have 
		\begin{align}\label{new1}
			M_{ \rho,  \Lambda}^{m}\left(\mathbb{Z}^{n} \times \mathbb{T}^{n}\right) \subset S_{ \rho, \Lambda}^{m}\left(\mathbb{Z}^{n} \times \mathbb{T}^{n}\right) .
		\end{align}
	\end{remark}
	
	\begin{proposition}\label{eq5}
		Let $\sigma(k, x)\in  C^{\infty} \left( \mathbb{Z}^n \times \mathbb{T}^n \right) $. Then the following  statements are equivalent:
		\begin{enumerate}
			\item  $  \sigma(k, x) \in M_{\rho,  \Lambda}^{m} \left( \mathbb{Z}^n \times \mathbb{T}^n \right),$
			\item \label{eq2}   for every $\alpha \in \mathbb{Z}_{+}^{n}$  and  $\gamma \in\{0,1\}^n$, we have 
			\begin{align} \label{eq1}
				k^\gamma \Delta_k^{\alpha+\gamma} \sigma(k, x)\in S_{ \rho, \Lambda}^{m-\rho|\alpha|}\left( \mathbb{Z}^n \times \mathbb{T}^n \right),
			\end{align}
			\item for every $\alpha, \beta \in \mathbb{Z}_{+}^{n}$, there  exists a constant $C_{ \alpha, \beta}>0$ such that 
			\begin{align} \label{eq3}
				\left| k^\gamma D_{x}^{(\beta)} \Delta_{k}^{\alpha+\gamma} \sigma(k, x)\right| \leq C_{ \alpha, \beta, \gamma}~\Lambda(k)^{m-\rho|\alpha|}	\end{align}	
			for all  $ k \in \mathbb{ Z}^n,  x \in \mathbb{T}^n$ and $ \gamma \in\{0,1\}^n$.
		\end{enumerate}
	\end{proposition}
	\begin{proof}
		(1)$\implies$ (2):	Let us assume that $ \sigma(k, x) \in M_{\rho,  \Lambda}^{m}\left( \mathbb{Z}^n \times \mathbb{T}^n \right)$. In order to prove (2), we will use induction on $|\alpha|$. When $|\alpha|=0$, then (\ref{eq1}) holds trivially. Let us assume that (\ref{eq1})  holds  for  $|\alpha|\leq \ell .$ Let $\alpha'=\alpha+e_j$. Then using Leibnitz rule, we have 
		\begin{align*}
			k^\gamma 	  \Delta_k^{\alpha'+\gamma} \sigma(k, x) &=		\Delta_{k_{j}} 	( k^\gamma \Delta_k^{\alpha+\gamma} \sigma(k, x))-( \Delta_{k_{j}} 	 k^\gamma) \Delta_k^{\alpha+\gamma} \sigma(k, x).
		\end{align*}
		From the inductive hypothesis $$\Delta_{k_{j}} 	( k^\gamma \Delta_k^{\alpha+\gamma} \sigma(k, x))\in   S_{ \rho, \Lambda}^{m-\rho|\alpha|-\rho}\left( \mathbb{Z}^n \times \mathbb{T}^n \right)=S_{ \rho, \Lambda}^{m-\rho|\alpha'|}\left( \mathbb{Z}^n \times \mathbb{T}^n \right).$$
		Moreover, when $\gamma_j=0$,  then $( \Delta_{k_{j}} 	 k^\gamma) \Delta_k^{\alpha+\gamma} \sigma(k, x)=0$.  Again when  $\gamma_j=1$,  then  the $j$-th entry of $\gamma-e_j=\gamma'$ vanishes. Thus, using the inductive assumption, we have 
		\begin{align*}
			( \Delta_{k_{j}} 	 k^\gamma) \Delta_k^{\alpha+\gamma} \sigma(k, x)&= k^{\gamma'}\Delta_k^{\alpha+\gamma} \sigma(k, x)\\
			&=k^{\gamma'} \Delta_k^{\alpha'+\gamma'} \sigma(k, x)\\
			&=\Delta_{k_{j}}( k^{\gamma'}\Delta_k^{\alpha + \gamma'} \sigma(k, x))\in S_{ \rho, \Lambda}^{m-\rho|\alpha|-\rho}\left( \mathbb{Z}^n \times \mathbb{T}^n \right)=S_{ \rho, \Lambda}^{m-\rho|\alpha'|}\left( \mathbb{Z}^n \times \mathbb{T}^n \right).
		\end{align*}
		This completes the proof of part (2).  
		
		(2)$\implies$ (3): Since 
		$
		k^\gamma \Delta_k^{\alpha+\gamma} \sigma(k, x)\in S_{ \rho, \Lambda}^{m-\rho|\alpha|}\left( \mathbb{Z}^n \times \mathbb{T}^n \right)
		$   for any $\alpha \in \mathbb{Z}_{+}^{n}$  and  $\gamma \in\{0,1\}^n$, then 	\begin{align*} 
			\left| k^\gamma D_{x}^{(\beta)} \Delta_{k}^{\alpha+\gamma} \sigma(k, x)\right| \leq C_{ \alpha, \beta,\gamma}~\Lambda(k)^{m-\rho|\alpha|}	\end{align*}	
		for all  $ k \in \mathbb{ Z}^n,  x \in \mathbb{T}^n$ and $ \gamma \in\{0,1\}^n.$ This completes proof of part (3).

		(3)$\implies$ (1):  Now, we assume that   $\sigma(k, x)$ satisfies the estimates (\ref{eq3}).  Then   for $\gamma=0$, we obtain $\sigma(k, x)\in   S_{ \rho, \Lambda}^{m}\left( \mathbb{Z}^n \times \mathbb{T}^n \right)$. Thus for all  $\gamma \in\{0,1\}^n$ and $\alpha, \beta \in \mathbb{Z}_+^n$, using Leibnitz formula, we obtain
		\begin{align*} 
			D_{x}^{(\beta)} \Delta_{k}^{\alpha} (	k^\gamma \Delta_k^{\gamma} \sigma(k, x))  &= \sum_{\eta \leq \alpha}C_{\alpha, \eta, \gamma} \;  k^{\gamma-\eta} \;D_{x}^{(\beta)}  {\Delta}_{k}^{\alpha-\eta+\gamma} \sigma(k+\eta, x).
		\end{align*}
		An application of Peetre’s inequality (see Proposition 3.3.31 \cite{ruz1}) gives   
		\begin{align*} 
			| D_{x}^{(\beta)} \Delta_{k}^{\alpha} (	k^\gamma \Delta_k^{\gamma} \sigma(k, x)) | &\leq C_{\alpha, \beta, \gamma}~ \Lambda(k)^{m-\rho|\alpha|}, 
		\end{align*}
		for all  $ k \in \mathbb{ Z}^n,  x \in \mathbb{T}^n$,  i.e.,   $k^\gamma \Delta_k^{\gamma} \sigma(k, x)\in S_{ \rho, \Lambda}^{m }\left( \mathbb{Z}^n \times \mathbb{T}^n \right)$. This implies that $  \sigma(k, x) \in M_{\rho,  \Lambda}^{m}\left( \mathbb{Z}^n \times \mathbb{T}^n \right)$ and completes the proof of the proposition.
	\end{proof}
	
	The following two statements follows immediately as an application of the above results. 
	
	%	Using the relation (\ref{s2}) and (\ref{s3}), we have the following results.
	
	\begin{proposition}\label{new3}
		Let $m_1, m_2 \in \mathbb{R}$ and $\rho\in (0, \frac{1}{\mu}]$. Then for  $  a(k, x) \in M_{\rho,  \Lambda}^{m_1}\left(\mathbb{Z}^{n} \times \mathbb{T}^{n}\right) $ and $b(k, x) \in M_{\rho,  \Lambda}^{m_2}\left(\mathbb{Z}^{n} \times \mathbb{T}^{n}\right) $,  the following properties hold:
		\begin{enumerate}
			\item If $m_1 \leq m_2$ then $M_{\rho, \Lambda}^{m_1}\left(\mathbb{Z}^{n} \times \mathbb{T}^{n}\right) \subset M_{\rho,  \Lambda}^{m_2}\left(\mathbb{Z}^{n} \times \mathbb{T}^{n}\right) $;
			\item $a(k, x) b(k, x) \in M_{\rho, \Lambda}^{m_1+m_2}\left(\mathbb{Z}^{n} \times \mathbb{T}^{n}\right) $;
			\item  $  D_{x}^{(\beta)}  \Delta_{k}^{\alpha}a(k, x)  \in M_{\rho,  \Lambda}^{m_1-\rho |\alpha|}\left(\mathbb{Z}^{n} \times \mathbb{T}^{n}\right) $ for all  $\alpha, \beta \geq 0$.
		\end{enumerate}
	\end{proposition}
	\begin{proof}
		Let $a(k, x)\in   M_{\rho,  \Lambda}^{m_1}\left(\mathbb{Z}^{n} \times \mathbb{T}^{n}\right) $. Then for all  $\gamma \in\{0,1\}^n$, we have 
		$
		k^\gamma \Delta_k^{\gamma} a(k, x)\in S_{ \rho, \Lambda}^{m}\left( \mathbb{Z}^n \times \mathbb{T}^n \right) 
		,$ i.e.,  for all  multi-indices $\alpha, \beta \in \mathbb{Z}^n $, there exists a positive constant $C_{\alpha, \beta}$ such that 
		\begin{align}\label{eq4}
			\left| D_{x}^{(\beta)} \Delta_{k}^{\alpha}\left(k^\gamma \Delta_k^{\gamma}a(k, x)\right)\right| \leq C_{ \alpha, \beta, \gamma}~\Lambda(k)^{m_1-\rho|\alpha|}.
		\end{align}
		Since $m_1 \leq m_2$, from the relation  (\ref{eq4}),   we get 
		$$\left| D_{x}^{(\beta)} \Delta_{k}^{\alpha}\left(k^\gamma \Delta_k^{\gamma}a(k, x)\right)\right| \leq C_{ \alpha, \beta, \gamma}~\Lambda(k)^{m_2-\rho|\alpha|}$$
		for all $k \in \mathbb{Z}^{n}$ and $x \in \mathbb{T}^{n}$. This implies that  $a\in   M_{\rho,  \Lambda}^{m_2}\left(\mathbb{Z}^{n} \times \mathbb{T}^{n}\right) $. Again, let $  a(k, x) \in M_{ \rho, \Lambda}^{m_1}\left(\mathbb{Z}^{n} \times \mathbb{T}^{n}\right) $ and $b(k, x) \in M_{\rho,  \Lambda}^{m_2}\left(\mathbb{Z}^{n} \times \mathbb{T}^{n}\right) $. In order to show   $(ab)(k, x) \in M_{\rho, \Lambda}^{m_1+m_2}\left(\mathbb{Z}^{n} \times \mathbb{T}^{n}\right) $, by part (3) of Proposition \ref{eq5}, its enough to prove that,  for every $\alpha, \beta \in \mathbb{Z}_{+}^{n}$, there  exists a constant $C_{ \alpha, \beta}>0$ such that 
		\begin{align} \label{eq300}
			\left| k^\gamma D_{x}^{(\beta)} \Delta_{k}^{\alpha+\gamma} (ab)(k, x)\right| \leq C_{ \alpha, \beta,\gamma}~\Lambda(k)^{m-\rho|\alpha|}	\end{align}	
		for all  $ k \in \mathbb{ Z}^n,  x \in \mathbb{T}^n$ and $ \gamma \in\{0,1\}^n$.  Now,  for any $\gamma \in\{0,1\}^n$, we have
		\begin{align}\label{eq301}\nonumber
			&	k^\gamma D_{x}^{(\beta)}  \Delta_k^{\alpha+\gamma} (ab)(k, x)\\\nonumber
			&=\sum_{\tilde{\alpha} \leq \alpha+\gamma}C_{ \tilde{\alpha}, \gamma}~D_{x}^{(\beta)}  (k^{{\tilde{\alpha}}} \Delta_{k}^{\tilde{\alpha}} a(k, x)) ~(k^{\gamma-\tilde{\alpha}}~ {\Delta}_{k}^{\alpha+\gamma-\tilde{\alpha}} b(k+\tilde{\alpha}, x))\\
			&=\sum_{\tilde{\alpha} \leq \alpha+\gamma}C_{ \tilde{\alpha}, \gamma} \sum_{\tilde{\beta} \leq \beta} C_{\tilde{\beta}}   (k^{{\tilde{\alpha}}}D_{x}^{(\tilde{\beta})} \Delta_{k}^{\tilde{\alpha}} a(k, x)) ~(k^{\gamma-\tilde{\alpha}}~ D_{x}^{(\beta-\tilde{\beta})}{\Delta}_{k}^{\alpha+\gamma-\tilde{\alpha}} b(k+\tilde{\alpha}, x)).
		\end{align}
		Since $  a(k, x) \in M_{ \rho, \Lambda}^{m_1}\left(\mathbb{Z}^{n} \times \mathbb{T}^{n}\right) $ and $b(k, x) \in M_{\rho,  \Lambda}^{m_2}\left(\mathbb{Z}^{n} \times \mathbb{T}^{n}\right) $,  from Proposition \ref{eq5},  (\ref{eq300}), and  (\ref{eq301}), we have  
		%$k^\beta \Delta_{k}^{\beta} a (k, x)\in S_{\rho,  \Lambda}^{m_1}  \left(\mathbb{Z}^{n} \times \mathbb{T}^{n}\right) $ and $k^{\gamma-\beta}  {\Delta}_{k}^{\gamma-\beta} b(k, x)\in S_{\rho, \Lambda}^{m_2}(\mathbb{Z}^n \times \mathbb{T}^n)  $.	This implies that 
		%$k^\gamma \Delta_k^{\gamma} (ab)(k, x) \in S_{\rho, \Lambda}^{m_1+m_2}\left(\mathbb{Z}^{n} \times \mathbb{T}^{n}\right) .$
		$a(k, x)b(k, x) \in M_{\rho, \Lambda}^{m_1+m_2}\left(\mathbb{Z}^{n} \times \mathbb{T}^{n}\right) .$ This completes   statement (2). Statement (3)  follows from  Proposition \ref{eq5}  and the characterization given in  (\ref{s2}).
	\end{proof}
	\begin{corollary}
		If $\Lambda(k)$ is a weight function. Then for any $m \in \mathbb{R}$, $\Lambda(k) \in M_{\frac{1}{\mu}, \Lambda}^{1}\left(\mathbb{Z}^{n} \times \mathbb{T}^{n}\right) $ and $\Lambda(k)^{m} \in M_{\frac{1}{\mu}, \Lambda}^{m}\left(\mathbb{Z}^{n} \times \mathbb{T}^{n}\right) $. In particular, if $\rho\in (0, \frac{1}{\mu}]$, then $\Lambda(k) \in M_{\rho, \Lambda}^{1}\left(\mathbb{Z}^{n} \times \mathbb{T}^{n}\right) $ and $\Lambda(k)^{m} \in M_{\rho,  \Lambda}^{m}\left(\mathbb{Z}^{n} \times \mathbb{T}^{n}\right) $.
	\end{corollary}
	
	\begin{lemma}\label{new2}
		For  every $m \in \mathbb{R}$ and $0<\rho\leq \frac{1}{\mu}$, there exists a positive   $N_{0}$ such that
		\begin{align}\label{eq7}
			S_{\rho, \Lambda}^{m-N_{0}} \left(\mathbb{Z}^{n} \times \mathbb{T}^{n}\right) \subset M_{ \rho, \Lambda}^{m} \left(\mathbb{Z}^{n} \times \mathbb{T}^{n}\right) \subset S_{ \rho, \Lambda}^{m}\left(\mathbb{Z}^{n} \times \mathbb{T}^{n}\right) .
		\end{align}
		More precisely, \(N_{0}:=n\left(\frac{1}{\mu_{0}}-\rho\right)\). Moreover,  \begin{align}\label{eq8}
			\bigcap_{m \in \mathbb{R}} M_{\rho, \Lambda}^{m}\left( \mathbb{Z}^n \times \mathbb{T}^n \right) =\bigcap_{m \in \mathbb{R}} S_{\rho, \Lambda}^{m}\left( \mathbb{Z}^n \times \mathbb{T}^n \right) = S_{\rho, \Lambda}^{-\infty}\left( \mathbb{Z}^n \times \mathbb{T}^n \right) .
		\end{align}
	\end{lemma}
	\begin{proof}
		The inclusion relation on right hand side in (\ref{eq7})  follows easily from Remark \ref{eq6}. 	For the left hand side inclusion relation, take $\sigma(k, x) \in S_{ \rho, \Lambda}^{m-N_{0}}\left(\mathbb{Z}^{n} \times \mathbb{T}^{n}\right) $. Then,    for all multi-indices $\alpha, \beta ,  \gamma \geq 0 $   with  $\gamma_j \in\{0,1\},$ we have
		$$
		\left| k^\gamma D_{x}^{(\beta)}   \Delta_{k}^{\alpha+\gamma} \sigma(k, x)\right| \leq C_{\alpha, \beta}~(1+|k| )^{|\gamma|}\; \Lambda(k)^{m-N_{0}-\rho|\alpha+\gamma|}, \quad k\in\mathbb{Z}^n, x\in  \mathbb{T}^n
		$$
		for    some positive constant $C_{\alpha, \beta }$.  On the other hand, using the polynomial growth of $\Lambda(k)$ (see \eqref{growth}), we have
		$$
		(1+|k|)^{|\gamma|} \leq C_{\gamma}~ \Lambda(k)^{\frac{1}{\mu_{0}}|\gamma|}, \quad k\in\mathbb{Z}^n.
		$$
		Since $\rho \leq \frac{1}{\mu} \leq \frac{1}{\mu_{0}}$ and $|\gamma|\leq n$, we have 
		\begin{align*}\left| k^\gamma D_{x}^{(\beta)}   \Delta_{k}^{\alpha+\gamma} \sigma(k, x)\right| 
			& \leq C_{\alpha, \beta, \gamma}^{\prime} ~\Lambda(k)^{m-N_{0}+\left(\frac{1}{\mu_{0}}-\rho\right)|\gamma|-\rho|\alpha|} \\ 
			& \leq C_{\alpha, \beta, \gamma}^{\prime}~ \Lambda(k)^{m-\rho|\alpha|}.
		\end{align*} 
		Thus,  in view of  Proposition \ref{eq5},  we have $ \sigma(k, x) \in M_{\rho, \Lambda}^{m}\left(\mathbb{Z}^{n} \times \mathbb{T}^{n}\right) $.  Finally, the relation (\ref{eq8}) follows from the inclusions given in (\ref{eq7}).
	\end{proof}

	\section{Symbolic calculus and parametrix for $M_{\rho, \Lambda}^{m}\left( \mathbb{Z}^n \times \mathbb{T}^n \right) $}\label{sec3}
	In this section, we study  elements of the symbolic calculus of pseudo-differential operators on $\mathbb{Z}^n$ associated with the symbol class $M_{\rho, \Lambda}^{m}\left( \mathbb{Z}^n \times \mathbb{T}^n \right) $ by deriving formulae for the composition, adjoint, transpose. At the end of this section, we   derive formulae for  the parametrix for $M$-elliptic pseudo-differential operators. We begin with  the following result related with the asymptotic sum of symbols.
	\begin{theorem}\label{s5}
		Let $\{m_j\}_{j\in \mathbb{N}_0}$be a strictly decreasing sequence of real
		numbers such that $m_{j} \rightarrow-\infty
		$
		as $j \rightarrow \infty .$  Suppose that $\{\sigma_{j}\}_{j \in \mathbb{N}_0} \subset M_{\rho,  \Lambda}^{m_{j}}\left(\mathbb{Z}^n\times \mathbb{T}^n\right), j\in \mathbb{N}_0$ is a sequence of symbols. Then there exists a symbol $\sigma \in M_{\rho,  \Lambda}^{m_{0}}\left( \mathbb{Z}^n \times \mathbb{T}^n \right)$ such that
		$$
		\sigma(k, x) \sim \sum_{j=0}^{\infty} \sigma_{j}(k, x), 
		$$
		that means, 
		$$
		\sigma(k, x) -\sum_{j=0}^{N-1} \sigma_{j}(k, x)  \in M_{\rho,  \Lambda}^{m_{N}}\left(\mathbb{Z}^n\times \mathbb{T}^n\right)
		$$
		for every positive integer $N .$ %Moreover, if $\tau$ is another symbol with the same asymptotic expansion, then $\sigma-\tau \in$ $\bigcap_{m \in \mathbb{R}} M_{\rho, \Lambda}^{m}$
	\end{theorem}
	\begin{proof}
		Let $\sigma_{j} \in M_{\rho,  \Lambda}^{m_{j}}\left(\mathbb{Z}^n\times \mathbb{T}^n\right)$. 		Then  from the relation   (\ref{new1}),  we have $\sigma_{j} \in S_{\rho,  \Lambda}^{m_{j}}\left(\mathbb{Z}^n\times \mathbb{T}^n\right)$.   We   consider  a function $\phi \in C^{\infty}\left(\mathbb{R}^{n}\right)$ such that 
		
		$$	\phi (\xi)=\begin{cases}
			1& \text{if } |\xi| \geq 1,\\
			0&\text{if } |\xi| \leq \frac{1}{2} \\
		\end{cases}$$
		and 	$0\leq \phi(\xi )\leq 1 $  otherwise. Let  $\left(\varepsilon_{j}\right)_{j=0}^{\infty}$  be a sequence of positive real numbers such that $\varepsilon_{j}>\varepsilon_{j+1} \rightarrow 0$. Define  the function  $\phi _{j} \in C^{\infty}\left(\mathbb{R}^{n}\right)$
		by $\phi _{j}(\xi):=\phi  (\varepsilon_{j} \xi ) .$  
		The support   of $\Delta_{\xi}^{\alpha} \phi _{j}$ is bounded when   $|\alpha| \geq 1$. 
		Since $\sigma_{j} \in S_{\rho,  \Lambda}^{m_{j}}\left(\mathbb{Z}^n\times \mathbb{T}^n\right)$, using  the discrete Leibniz formula, we get 
		$$
		\left| D_{x}^{(\beta)} \Delta_{\xi}^{\alpha}  \left[\phi _{j}(\xi) \sigma_{j}(\xi, x)\right]\right| \leq C_{j \alpha \beta}\Lambda(\xi)^{m_{j}-\rho|\alpha| },
		$$
		where $C_{j \alpha \beta}$ is a positive constant.  This shows  that
		$ \phi _{j}(\xi) \sigma_{j}(\xi, x) \in S_{\rho,  \Lambda}^{m_{j}}\left(\mathbb{Z}^n\times \mathbb{T}^n\right) .$ 
		Further,   when $j$ is large
		enough, 	%Examining the support of $\Delta_{\xi}^{\alpha} \phi _{j}$, 
		$\Delta_{\xi}^{\alpha}\left(\phi _{j}(\xi) \sigma_{j}(\xi, x)\right)$ vanishes for any fixed $\xi \in \mathbb{Z}^{n}$.  Construct $\sigma(\xi, x)$ of the form 
		$$
		\sigma(\xi, x):=\sum_{j=0}^{\infty} \phi _{j}(\xi) \sigma_{j}(\xi, x) \quad \xi \in\mathbb{Z}^n, x\in  \mathbb{T}^n.
		$$
		Clearly  $\sigma \in S_{\rho,  \Lambda}^{m_0}\left(\mathbb{Z}^n\times \mathbb{T}^n\right) .$  Further, we have 
		\begin{align*}
			&\bigg|D_{x}^{(\beta)} \Delta_{\xi}^{\alpha} \bigg[\sigma(\xi, x)-\sum_{j=0}^{N-1} \sigma_{j}(\xi, x)\bigg]\bigg| \\
			&\leq \sum_{j=0}^{N-1}\left|D_{x}^{(\beta)} \Delta_{\xi}^{\alpha}  \left\{\left[\phi _{j}(\xi)-1\right] \sigma_{j}(\xi, x)\right\}\right|+\sum_{j=N}^{\infty}\left|D_{x}^{(\beta)} \Delta_{\xi}^{\alpha} \left[\phi _{j}(\xi) \sigma_{j}(\xi, x)\right]\right|.
		\end{align*}
		Since $\varepsilon_{j}>\varepsilon_{j+1}$ and $\varepsilon_{j} \rightarrow 0$ as $j \rightarrow \infty,$  then the  sum $\sum_{j=0}^{N-1}\left|D_{x}^{(\beta)} \Delta_{\xi}^{\alpha}  \left\{\left[\phi _{j}(\xi)-1\right] \sigma_{j}(\xi, x)\right\}\right|$   vanishes, whenever $|\xi|$ is large. Thus, there exists a  positive constant  $C_{r N \alpha \beta}$ such that  $$\sum_{j=0}^{N-1}\left|D_{x}^{(\beta)} \Delta_{\xi}^{\alpha}  \left\{\left[\phi _{j}(\xi)-1\right] \sigma_{j}(\xi, x)\right\}\right|\leq C_{r N \alpha \beta}\Lambda(\xi)^{-r}$$ for any $r \in \mathbb{R}$. On the other side, a routine  calculation yields  $$\sum_{j=N}^{\infty}\left|D_{x}^{(\beta)} \Delta_{\xi}^{\alpha} \left[\phi _{j}(\xi) \sigma_{j}(\xi, x)\right]\right|\leq C_{N \alpha \beta}^{\prime}\Lambda(\xi)^{m_{N}-\rho|\alpha| },$$
		%	Then  from the relation   (\ref{new1}),  $\sigma_{j} \in S_{\rho,  \Lambda}^{m_{j}}\left(\mathbb{Z}^n\times \mathbb{T}^n\right)$.  Using  Lemma 3.4 of  \cite{bot},  there exists a symbol $\sigma$ in $S_{ \rho, \Lambda}^{m_{0}}$ such that $$ 	\sigma \sim \sum_{j=0}^{\infty} \sigma_{j} 	$$
		where  $C_{N \alpha \beta}^{\prime}$ is a positive constant. This shows that for every $N\in \mathbb{N}$, we have
		$$
		\sigma(k, x)-\sum_{j=0}^{N-1}  \sigma_{j}(k, x) \in S_{\rho, \Lambda}^{m_{N}}\left(\mathbb{Z}^{n} \times \mathbb{T}^{n}\right) .
		$$
		Since $m_{j} \rightarrow-\infty,$ as $j \rightarrow \infty,$ using   left
		inclusions in Lemma \ref{new2},  we have   $\displaystyle \sigma-\sum_{j=0}^{N-1} \sigma_{j} \in S_{\rho, \Lambda}^{m_N} \left(\mathbb{Z}^{n} \times \mathbb{T}^{n}\right) \subset M_{\rho,  \Lambda}^{m_0}\left(\mathbb{Z}^{n} \times \mathbb{T}^{n}\right) $ for a sufficiently	large \(N .\) Then we get \(\sigma(k, x) \in M_{\rho, \Lambda}^{m_0} \left(\mathbb{Z}^{n} \times \mathbb{T}^{n}\right) .\) Furthermore, for all \(N \geq 2\) and \(N^{\prime} \geq N\)
		$$
		\sigma-\sum_{j=0}^{N-1} \sigma_{j} =\sum_{j=N}^{N'-1} \sigma_{j}+r_{N'}
		$$
		with $r_{N'} \in S_{ \rho, \Lambda}^{m_{N'}}\left(\mathbb{Z}^{n} \times \mathbb{T}^{n}\right) $.  By choosing a sufficiently large $N'$ so that $m_{N'}<m_{N}-N_{0}$,
		we have $r_{N'} \in S_{ \rho,  \Lambda}^{m_{N}-N_0}\left(\mathbb{Z}^{n} \times \mathbb{T}^{n}\right) \subset M_{ \rho,  \Lambda}^{m_{N}}\left(\mathbb{Z}^{n} \times \mathbb{T}^{n}\right) $ and therefore  $\displaystyle \sigma-\sum_{j=0}^{N-1} \sigma_{j}  \in M_{\rho, 
			\Lambda}^{m_{N}}\left(\mathbb{Z}^{n} \times \mathbb{T}^{n}\right) $.  This completes the proof of the theorem.
	\end{proof}
	We now obtain the asymptotic formula composition of  two $M$-elliptic pseudo-differential operators.
	\begin{theorem} \label{4}
		Let $m_1, m_2 \in \mathbb{R}$. 	Let $\sigma  \in M_{\rho,  \Lambda}^{m_1}\left(\mathbb{Z}^{n} \times \mathbb{T}^{n}\right) $ and $\tau  \in M_{\rho,  \Lambda}^{m_2}\left(\mathbb{Z}^{n} \times \mathbb{T}^{n}\right) $. Then the composition   $T_\sigma T_\tau$ of  the pseudo-differential operators $T_\sigma$  and $T_\tau $ is  a pseudo-differential operator  with symbol $\lambda\in M_{\rho,  \Lambda}^{m_1+m_2}\left(\mathbb{Z}^{n} \times \mathbb{T}^{n}\right) $ having  asymptotic expansion
		
		\begin{align*}\lambda(k, x)& {\sim \sum_{\alpha \geq 0} \frac{1}{\alpha !}\left(D_{x}^{(\alpha)} \sigma (k, x)\right) \Delta_k^{\alpha} \tau(k, x)},	\end{align*}	where the asymptotic expansion means that for every $N \in \mathbb{N}$, we have $$\lambda(k, x)-\sum_{|\alpha|<N} \frac{1}{\alpha !}\left(D_{x}^{(\alpha)} \sigma (k, x)\right) \Delta_k^{\alpha} \tau(k, x)\in M_{\rho,  \Lambda}^{m_1+m_2-\rho{N}}\left(\mathbb{Z}^n\times \mathbb{T}^n\right). $$\end{theorem} 
	
	\begin{proof}
		Let $T_\sigma$  and $T_\tau $ be two   pseudo-differential operators  with symbols $\sigma  \in M_{\rho,  \Lambda}^{m_1}\left(\mathbb{Z}^{n} \times \mathbb{T}^{n}\right) $ and $\tau  \in M_{\rho,  \Lambda}^{m_2}\left(\mathbb{Z}^{n} \times \mathbb{T}^{n}\right) $ respectively. % Then  by  the relation   (\ref{new1}),  we have  $\sigma  \in S_{\rho,  \Lambda}^{m_1}\left(\mathbb{Z}^{n} \times \mathbb{T}^{n}\right) $ and $\tau  \in S_{\rho,  \Lambda}^{m_2}\left(\mathbb{Z}^{n} \times \mathbb{T}^{n}\right) $. Let $f, g \in \mathcal{S}\left(\mathbb{Z}^{n}\right)$. 
		Then, for  $f, g \in \mathcal{S}\left(\mathbb{Z}^{n}\right)$,  the pseudo-differential operators associated with  symbols $\sigma$ and $\tau$ are given by
		$$
		T_\sigma f(k)=\sum_{m \in \mathbb{Z}^{n}} \int_{\mathbb{T}^{n}} e^{2 \pi i(k-m) \cdot x} \sigma(k, x) f(m) \;dx$$
		and $$	T_\tau g(m)=\sum_{l \in \mathbb{Z}^{n}} \int_{\mathbb{T}^{n}} e^{2 \pi i(m-l) \cdot y} \tau(m, y) g(l) \;dy.
		$$
		Moreover, their composition  is given by 
		$$
		\begin{aligned}
			(T_\sigma T_\tau g)(k)=	T_\sigma (	T_\tau g)(k) &=\sum_{l \in \mathbb{Z}^{n}} \sum_{m \in \mathbb{Z}^{n}} \int_{\mathbb{T}^{n}} \int_{\mathbb{T}^{n}} e^{2 \pi i(k-m) \cdot x} \sigma(k, x) e^{2 \pi i(m-l) \cdot y} \tau(m, y) g(l) \;dy\;dx \\
			&=\sum_{l \in \mathbb{Z}^{n}} \int_{\mathbb{T}^{n}} e^{2 \pi i(k-l) \cdot y} \lambda(k, y) g(l) \;dy,
		\end{aligned}
		$$
		where the symbol $\lambda$ is given by 
		\begin{align}\label{eq9}\nonumber
			\lambda (k, y) &=\sum_{m \in \mathbb{Z}^{n}} \int_{\mathbb{T}^{n}} e^{2 \pi i(k-m) \cdot(x-y)} \sigma(k, x) \tau(m, y) \;dx \\\nonumber
			&=\sum_{m \in \mathbb{Z}^{n}} \int_{\mathbb{T}^{n}} e^{2 \pi i k \cdot(x-y)} e^{-2 \pi i m \cdot(x-y)} \sigma(k, x) \tau(m, y)  \;dx\\\nonumber
			&	=\int_{\mathbb{T}^{n}} e^{2 \pi i k \cdot(x-y)} \sigma(k, x) \widehat{\tau}(x-y, y) \mathrm{d} x \\
			&	=\int_{\mathbb{T}^{n}} e^{2 \pi i k \cdot x} \sigma(k, y+x) \widehat{\tau}(x, y) \;dx.
		\end{align}
		Here  $\widehat{\tau}$ denotes  the Fourier transform of $\tau(m, y)$ in the first variable. Using the periodic Taylor expansion (Theorem 3.4 of \cite{ruz1}),  we obtain
		$$
		\begin{aligned}
			\lambda(k, x) &=\int_{\mathbb{T}^{n}} e^{2 \pi i k \cdot y} \sigma(k, x+y) \widehat{\tau}(y, x) \;d y \\
			&=\int_{\mathbb{T}^{n}} e^{2 \pi i k \cdot y} \sum_{|\alpha|\geq 0} \frac{1}{\alpha !}\left(e^{2 \pi i y}-1\right)^{\alpha} D_{x}^{(\alpha)} \sigma(k, x) \widehat{\tau}(y, x)  \;d  y  \\
			&=\sum_{|\alpha|\geq 0} \frac{1}{\alpha !} D_{x}^{(\alpha)} \sigma(k, x) \Delta_{k}^{\alpha} \tau(k, x).
		\end{aligned}
		$$
		%	where $R$ is  the remainder term coming from the Taylor expansion formula.  From   Proposition \ref{eq5}, we have 	$$	D_{x}^{(\alpha)} \sigma(k, x) \Delta_{k}^{\alpha} \tau(k, x) \in S_{\rho, \delta}^{m_{1}+m_{2}-\rho |\alpha|}\left(\mathbb{Z}^{n} \times \mathbb{T}^{n}\right). 	$$ 	The remainder $R$ can be treated along the lines of the remainder treatment in the proof of Theorem 2.8. $\square$
		%	By  Theorem 3.1 of \cite{bot},  $T_\sigma T_\tau=T_\lambda$  is pseudo-differential operator  with symbol $\lambda\in S_{\rho,  \Lambda}^{m_1+m_2}\left(\mathbb{Z}^{n} \times \mathbb{T}^{n}\right) $. Moreover,  its   asymptotic expansion is given by 	$$\lambda(k, x){\sim \sum_{\alpha \geq 0} \frac{1}{\alpha !} D_{x}^{(\alpha)} \sigma (k, x) \Delta_k^{\alpha} \tau(k, x)}.$$
		Since  $ \sigma(k, x) \in M_{ \rho, \Lambda}^{m_1}\left(\mathbb{Z}^{n} \times \mathbb{T}^{n}\right) $ and $\tau(k, x) \in M_{\rho,  \Lambda}^{m_2}\left(\mathbb{Z}^{n} \times \mathbb{T}^{n}\right) $,  from Proposition \ref{new3}, for every non negative  integer $N$,  we have 
		$$\sum_{|\alpha|=N } \frac{1}{\alpha !} D_{x}^{(\alpha)} \sigma (k, x) \Delta_k^{\alpha} \tau(k, x)\in M_{\rho,  \Lambda}^{m_1+m_2-\rho N}\left(\mathbb{Z}^{n} \times \mathbb{T}^{n}\right) .$$
		Considering $\sum_{|\alpha|=N } \frac{1}{\alpha !} D_{x}^{(\alpha)} \sigma (k, x) \Delta_k^{\alpha} \tau(k, x)$ as a sequence $\{m_N\}$ and applying Theorem \ref{s5} on it, we get $\lambda(k, x){\sim \sum_{\alpha \geq 0} \frac{1}{\alpha !} D_{x}^{(\alpha)} \sigma (k, x) \Delta_k^{\alpha} \tau(k, x)}\in M_{\rho,  \Lambda}^{m_1+m_2}\left(\mathbb{Z}^{n} \times \mathbb{T}^{n}\right) .$ This completes  the proof of the theorem.
	\end{proof}
	
	The adjoint $T^*$ of the operator $T$ on $\ell^2(\mathbb{ Z }^n)$  is defined by
	\begin{align*}	
		\langle T_{\sigma} f, g \rangle_{\ell^2(\mathbb{ Z }^n)} =\langle f, T_{\sigma}^*  g \rangle_{\ell^2(\mathbb{ Z }^n)}	.\end{align*}	
	In the next result we derive an asymptotic formula for symbol of the adjoint operator.
	\begin{theorem} \label{s6}
		Let $m\in \mathbb{R}$ and $\sigma  \in M_{\rho,  \Lambda}^{m}\left(\mathbb{Z}^{n} \times \mathbb{T}^{n}\right) . $ Then, the adjoint    $T_\sigma ^*$ of    $T_\sigma$   is  a pseudo-differential operator $T_{\sigma^*} $,   where the symbol $\sigma^* \in M_{\rho,  \Lambda}^{m}\left(\mathbb{Z}^{n} \times \mathbb{T}^{n}\right) $ having  asymptotic expansion
		\begin{align*}\sigma^*(k, x)& \sim \sum_{\alpha \geq 0} \frac{1}{\alpha !} \Delta_k^{\alpha}  D_{x}^{(\alpha)} \overline{ \sigma(k, x)}.
		\end{align*}
		Here the asymptotic expansion means, that for every $N \in \mathbb{N}$, we have 
		$$\sigma^*(k, x)-\sum_{|\alpha|<N}  \frac{1}{\alpha !} \Delta_k^{\alpha}  D_{x}^{(\alpha)} \overline{ \sigma(k, x)}\in M_{\rho,  \Lambda} ^{m-\rho N}\left(\mathbb{Z}^{n} \times \mathbb{T}^{n}\right) . $$\end{theorem} 
	
	\begin{proof}
		From the definition of adjoint, we have 	\begin{align*}	
			\langle T_{\sigma} f, g \rangle_{\ell^2(\mathbb{ Z }^n)} =\langle f, T_{\sigma}^*  g \rangle_{\ell^2(\mathbb{ Z }^n)}	.\end{align*}	
		Now 	$$
		\begin{aligned}
			\langle T_{\sigma} f, g \rangle_{\ell^2(\mathbb{ Z }^n)}  &=\sum_{k \in \mathbb{Z}^{n}} \sum_{l \in \mathbb{Z}^{n}} \int_{\mathbb{T}^{n}} e^{2 \pi i(k-l) \cdot y} \sigma(k, y) f(l) \overline{g(k)} \;{d} y \\
			&=\sum_{l \in \mathbb{Z}^{n}} f(l)\left(\overline{\sum_{k \in \mathbb{Z}^{n}} \int_{\mathbb{T}^{n}} e^{-2 \pi i(k-l) \cdot y} \overline{\sigma(k, y)} g(k)} \;{d} y \right)=\langle f, T_{\sigma}^*  g \rangle_{\ell^2(\mathbb{ Z }^n)}.
		\end{aligned}
		$$
		Thus 
		$$
		T_{\sigma}^*g(k)=\sum_{l\in \mathbb{Z}^{n}} \int_{\mathbb{T}^{n}} e^{2 \pi i(k-l) \cdot y} \overline{\sigma(l, y)} g(l)\; {d} y=	T_{\sigma^*}g(k).
		$$
		Using   the periodic Taylor expansion,  similarly as in    (\ref{eq9}),  the symbol $\sigma^*$ of $	T_{\sigma}^*$ is given by 
		\begin{align*}
			\sigma^*(k, x) &=\sum_{l \in \mathbb{Z}^{n}} \int_{\mathbb{T}^{n}} e^{2 \pi i(k-l) \cdot(y-x)} \overline{\sigma(l, y)}    \;dy\\
			%&=\sum_{l \in \mathbb{Z}^{n}} \int_{\mathbb{T}^{n}} e^{2 \pi ik \cdot(y-x)} \overline{\sigma(y-x, y)}    \;dy\\
			&\sim \sum_{\alpha \geq 0} \frac{1}{\alpha !} \Delta_k^{\alpha}  D_{x}^{(\alpha)} \overline{ \sigma(k, x)}.
		\end{align*}
		%	Let $T_\sigma$ be a pseudo-differential operator with symbol   $\sigma  \in M_{\rho,  \Lambda}^{m}\left(\mathbb{Z}^{n} \times \mathbb{T}^{n}\right) $. By  Theorem 3.2 of \cite{bot},  $T_\sigma ^*$  is pseudo-differential operator  with symbol $\tau \in S_{\rho,  \Lambda}^{m}\left(\mathbb{Z}^{n} \times \mathbb{T}^{n}\right) $. Moreover,  its   asymptotic expansion is given by 	\begin{align*}\tau(k, x)& \sim \sum_{\alpha \geq 0} \frac{1}{\alpha !} \Delta_k^{\alpha}  D_{x}^{(\alpha)} \overline{ \sigma(k, x)}. 	\end{align*}
		Since  $ \sigma (k, x) $ is in $M_{ \rho, \Lambda}^{m}\left( \mathbb{Z}^n \times \mathbb{T}^n \right)$,     form Proposition \ref{new3}, for every non negative  integer $N$,  we have 
		$$\sum_{|\alpha|=N } \frac{1}{\alpha !}\Delta_k^{\alpha}  D_{x}^{(\alpha)} \overline{ \sigma(k, x)} \in M_{\rho,  \Lambda}^{m-\rho N}\left(\mathbb{Z}^{n} \times \mathbb{T}^{n}\right) .$$
		Finally,   considering $\sum_{|\alpha|=N } \frac{1}{\alpha !} \Delta_k^{\alpha}  D_{x}^{(\alpha)} \overline{ \sigma(k, x)}$ as a sequence $\{m_N\}$ and applying Theorem \ref{s5} on it, we get $\sigma^*(k, x)\sim \sum_{\alpha \geq 0} \frac{1}{\alpha !} \Delta_k^{\alpha}  D_{x}^{(\alpha)} \overline{ \sigma(k, x)} \in M_{\rho,  \Lambda}^{m}\left(\mathbb{Z}^{n} \times \mathbb{T}^{n}\right) .$ This completes  the proof of the theorem.
	\end{proof}
	For $f, g \in \mathcal{S}\left(\mathbb{Z}^{n}\right)$, we recall that the transpose $T^{t}$ of a linear operator $T$ is given by the distributional duality
	$$
	\left\langle T^{t} f, g\right\rangle_{\ell^2(\mathbb{ Z }^n)} =\langle f, T g\rangle_{\ell^2(\mathbb{ Z }^n)} ,
	$$
	which means that  for all $k \in \mathbb{Z}^{n},$ we have
	$$
	\sum_{k \in \mathbb{Z}^{n}}\left(T^{t} f\right)(k) g(k)=\sum_{k \in \mathbb{Z}^{n}} f(k)(T g)(k).
	$$
	\begin{theorem}[Transpose operators] Let $T_\sigma$ be a pseudo-differential operator with symbol $\sigma \in M_{\rho, \Lambda}^{m}\left(\mathbb{Z}^{n} \times \mathbb{T}^{n}\right) , m\in \mathbb{R}.$ Then the transpose operator $T_\sigma^t$ of $T_\sigma$  is also a pseudo-differential operator with symbol $\sigma ^t\in M_{\rho, \Lambda}^{m}\left(\mathbb{Z}^{n} \times \mathbb{T}^{n}\right) $. Moreover,  it has the following assymptotic expression 
		
		$$
		\sigma^{t}(k, x) \sim \sum_{\alpha} \frac{1}{\alpha !} \Delta_{k}^{\alpha} D_{x}^{(\alpha)} \sigma(k,-x),
		$$
		where the asymptotic expansion means that for every $N \in \mathbb{N}$, we have 
		$$\sigma^t(k, x)-\sum_{|\alpha|<N}  \frac{1}{\alpha !} \Delta_k^{\alpha}  D_{x}^{(\alpha)}\sigma(k, -x)\in M_{\rho,  \Lambda} ^{m-\rho N}\left(\mathbb{Z}^{n} \times \mathbb{T}^{n}\right) . $$\end{theorem} 
	
	\begin{proof}
		
		From the definition of transpose, we have 		$$
		\sum_{k \in \mathbb{Z}^{n}}\left(T^{t} f\right)(k) g(k)=\sum_{k \in \mathbb{Z}^{n}} f(k)(T g)(k).
		$$
		Now 	$$
		\begin{aligned}
			\sum_{k \in \mathbb{Z}^{n}} f(k)(T g)(k)&=\sum_{k \in \mathbb{Z}^{n}} \sum_{l \in \mathbb{Z}^{n}} \int_{\mathbb{T}^{n}} f(k)e^{2 \pi i(k-l) \cdot y} \sigma(k, y) g(l)  \;{d} y \\
			&= \sum_{l \in \mathbb{Z}^{n}} g(l) \left( \sum_{k \in \mathbb{Z}^{n}} \int_{\mathbb{T}^{n}} f(k)e^{2 \pi i(k-l) \cdot y} \sigma(k, y)   \;{d} y\right)\\
			&=\sum_{l \in \mathbb{Z}^{n}}  g(l)\left(T^{t} f\right)(l).
		\end{aligned}
		$$
		Thus 
		\begin{align*}
			\left(T^{t} f\right)(l)&= \sum_{k \in \mathbb{Z}^{n}} \int_{\mathbb{T}^{n}} f(k)e^{2 \pi i(k-l) \cdot y} \sigma(k, y)   \;{d} y\\&	=
			\sum_{k \in \mathbb{Z}^{n}} \int_{\mathbb{T}^{n}} f(k)e^{2 \pi i(l-k) \cdot y} \sigma(k, -y)   \;{d} y=	T_{\sigma^t}g(l).
		\end{align*}
		Using   the periodic Taylor expansion,  similarly as in    (\ref{eq9}),  the symbol $\sigma^t$ of $	T_{\sigma}^t$ is given by 
		\begin{align*}
			\sigma^t(k, x) &=\sum_{l \in \mathbb{Z}^{n}} \int_{\mathbb{T}^{n}} e^{2 \pi i(k-l) \cdot(y-x)} {\sigma(l, -y)}    \;dy\\
			%&=\sum_{l \in \mathbb{Z}^{n}} \int_{\mathbb{T}^{n}} e^{2 \pi ik \cdot(y-x)} \overline{\sigma(y-x, y)}    \;dy\\
			&\sim \sum_{\alpha \geq 0} \frac{1}{\alpha !} \Delta_k^{\alpha}  D_{x}^{(\alpha)} { \sigma(k, -x)}.
		\end{align*}
		%	Let $T_\sigma$ be a pseudo-differential operator with symbol   $\sigma  \in M_{\rho,  \Lambda}^{m}\left(\mathbb{Z}^{n} \times \mathbb{T}^{n}\right) $. By  Theorem 3.3 of \cite{bot},  $T_\sigma ^t$  is pseudo-differential operator  with symbol $\sigma^t\in S_{\rho,  \Lambda}^{m}\left(\mathbb{Z}^{n} \times \mathbb{T}^{n}\right) $. Moreover,  its   asymptotic expansion is given by 	\begin{align*}\sigma^t(k, x)& \sim \sum_{\alpha \geq 0} \frac{1}{\alpha !} \Delta_k^{\alpha}  D_{x}^{(\alpha)} { \sigma(k, -x)}	\end{align*}
		Since  $ \sigma (k, x) $ is in $M_{ \rho, \Lambda}^{m}\left(\mathbb{Z}^{n} \times \mathbb{T}^{n}\right) $,      form Proposition \ref{new3}, for every non negative  integer $N$,  we  get 
		$$\sum_{|\alpha|=N } \frac{1}{\alpha !}\Delta_k^{\alpha}  D_{x}^{(\alpha)} { \sigma(k, -x)} \in M_{\rho,  \Lambda}^{m-\rho N}\left(\mathbb{Z}^{n} \times \mathbb{T}^{n}\right) .$$
		Finally, considering $\sum_{|\alpha|=N } \frac{1}{\alpha !} \Delta_k^{\alpha}  D_{x}^{(\alpha)}{ \sigma(k, -x)}$ as a sequence $\{m_N\}$ and applying Theorem \ref{s5} on it, we get $\sigma^t(k, x)\sim \sum_{\alpha \geq 0} \frac{1}{\alpha !} \Delta_k^{\alpha}  D_{x}^{(\alpha)}  { \sigma(k, -x)} \in M_{\rho,  \Lambda}^{m}\left(\mathbb{Z}^{n} \times \mathbb{T}^{n}\right) .$ This completes  the proof of the theorem.
	\end{proof}
	
	\begin{definition}
		[M-elliptic] A symbol $\sigma \in M_{\rho,  \Lambda}^{m}\left( \mathbb{Z}^n \times \mathbb{T}^n \right) $  is said to be $M$-elliptic if there exist positive constants $C$ and $R_1$ such that
		\begin{align*}
			|\sigma(k, x)| \geq C \Lambda(k)^{m},
		\end{align*}
		for all $ x\in \mathbb{T}^n$ and all $ k \in \mathbb{Z}^n$ with $|k|  \geq R_1.$
	\end{definition}

	\begin{proposition}\label{3}
		For $m_1, m_2\in \mathbb{R},$ let $\sigma \in M_{\rho, \Lambda }^{m_1}\left(\mathbb{Z}^{n} \times \mathbb{T}^{n}\right) $ and $\tau \in M_{\rho, \Lambda }^{m_2}\left(\mathbb{Z}^{n} \times \mathbb{T}^{n}\right) $ be $M$-elliptic symbols. Then $\frac{\sigma}{\tau} \in M_{\rho, \Lambda }^{m_1-m_2}\left(\mathbb{Z}^{n} \times \mathbb{T}^{n}\right) $ .  In particular, $\frac{1}{\tau} \in   M_{\rho, \Lambda }^{-m_2}\left(\mathbb{Z}^{n} \times \mathbb{T}^{n}\right) $.
	\end{proposition}
	\begin{proof}
		In order to prove  the theorem, 	it is enough to show    that for any $m_1, m_2 \in \mathbb{R}$,  there exists a positive constant $C_{\alpha, \beta, \gamma}$ such that
		$$
		\left|D_{x}^{(\beta)} \Delta_{k}^{\alpha}  \left(k^{\gamma} \Delta_{k}^{\gamma}\left(\frac{\sigma}{\tau}\right)\right)(k, x)\right| \leq C_{\alpha, \beta, \gamma} \Lambda(k)^{m_1-m_2-\rho|\alpha|}, \quad (k,  x) \in  \mathbb{Z}^{n}  \times \mathbb{T}^{n}
		$$
		for every multi-indices $\alpha, \beta \in \mathbb{N}_{0}^{n}$ and $\gamma \in\{0,1\}^{n}.$ 	Using the  Leibnitz formula,  we get  
		
		\begin{align}\label{s14}\nonumber
			\left|D_{x}^{(\beta)} \Delta_{k}^{\alpha}  \left(k^{\gamma} \Delta_{k}^{\gamma}\left(\frac{\sigma}{\tau}\right)\right)(k, x) \right| &=\left|  D_{x}^{(\beta)} \left(\sum_{\eta \leq \alpha}\left(\begin{array}{l}\alpha \\\eta\end{array}\right)\left(\Delta_{k}^{\eta} k^{\gamma}\right)\left({\Delta}_{k}^{\alpha-\eta} \Delta_{k}^{\gamma}\left(\frac{\sigma}{\tau}\right)\right)(k+\eta)\right)\right| \\\nonumber
			&=\left|  \sum_{\eta \leq \alpha}\left(\begin{array}{l}
				\alpha \\
				\eta
			\end{array}\right)\left(\Delta_{k}^{\eta} k^{\gamma}\right) \left(D_{x}^{(\beta)}  {\Delta}_{k}^{\alpha-\eta} \Delta_{k}^{\gamma}\left(\frac{\sigma}{\tau}\right)\right)(k+\eta)\right| \\
			%&??=\left|  \sum_{\eta \leq \alpha}\left(\begin{array}{l}\alpha \\\eta\end{array}\right)\left(\Delta_{k}^{\alpha} k^{\gamma}\right) \left(D_{x}^{(\beta)}  {\Delta}_{k}^{\alpha-\eta+\gamma}\left(\frac{\sigma}{\tau}\right)\right)(k+\eta)\right| \\
			& \leq \sum_{\eta \leq \alpha} \left( \begin{array}{l}
				\alpha \\
				\eta
			\end{array}\right)\left|\Delta_{k}^{\eta} k^{\gamma}\right|\left|\left(D_{x}^{(\beta)}  {\Delta}_{k}^{\alpha-\eta+\gamma}\left(\frac{\sigma}{\tau}\right)\right)(k+\eta)\right|.
		\end{align}
		From   the relation (3.14) of \cite{kal}, we have 
		$$ 	\left|\Delta_{k}^{\eta}k^{\gamma}\right| \leq C_{\alpha, \gamma} \Lambda(k)^{\rho|\gamma-\eta|}, \quad k \in \mathbb{Z}^n	$$ 	for every multi-indices $ \eta \in \mathbb{N}_0^n$ and $  \gamma \in \{0, 1\}^n$. 
		%Therefore from above we get \begin{align}\label{s11}	\left|D_{x}^{(\beta)} \Delta_{k}^{\alpha}  \left(k^{\gamma} \Delta_{k}^{\gamma}\left(\frac{\sigma}{\tau}\right)\right)(k, x) \right|  \leq C_{\alpha, \gamma}\sum_{\eta \leq \alpha} \left( \begin{array}{l}		\alpha \\	\eta	\end{array}\right) \Lambda(k)^{\rho|\alpha-\gamma|}\left|\left(D_{x}^{(\beta)}  {\Delta}_{k}^{\alpha-\eta+\gamma}\left(\frac{\sigma}{\tau}\right)\right)(k+\eta)\right|.		\end{align}
		The only part    remains  is to  estimate  $\left|\left(D_{x}^{(\beta)}  {\Delta}_{k}^{\alpha-\eta+\gamma}\left(\frac{\sigma}{\tau}\right)\right)(k+\eta)\right| $. By the property of   familiar differential operator we know that
		$$
		D_{x_i}\frac{\sigma(k, x)}{\tau(k, x)}=\frac{\tau(k, x) D_{x_i}\sigma(k, x)-\sigma(k, x) D_{x_i}\tau(k, x)}{\tau(k, x)^2}.
		$$
		The action of the difference operator is quite similar to the above and one can see that 
		$$
		\Delta_{k_{j}} \frac{\sigma(k, x)}{\tau(k, x)}=\frac{\tau(k, x) \Delta_{{k}_j} \sigma(k, x)-\sigma(k, x) \Delta_{k_j} \tau(k, x)}{\tau(k, x) \tau\left(k+v_{j}, x\right)}.
		$$
		Let us define the iterative sequence as following: for all multi-indices $|\alpha| \geq 1, |\beta| \geq 1$ we let  
		$$
		\left\{\begin{array}{l}
			\sigma_{\alpha, \beta+v_{i}}(k, x)=\tau_{\alpha, \beta}(k, x) D_{x_i}\sigma_{\alpha, \beta}(k, x)-\sigma_{\alpha, \beta}(k, x) D_{x_i}\tau_{\alpha, \beta}(k, x) \\
			\tau_{\alpha, \beta+v_{i}}(k, x)=\tau_{\alpha, \beta}(k, x)^{2} \\
			\sigma_{\alpha+v_{j}, \beta}(k, x)=\tau_{\alpha, \beta}(k, x) \Delta_{{k}_j} \sigma_{\alpha, \beta}(k, x)-\sigma_{\alpha, \beta}(k, x) \Delta_{{k}_j}\tau_{\alpha, \beta}(k, x) \\
			\tau_{\alpha+v_{j}, \beta}(k, x)=\tau_{\alpha, \beta}(k, x) \tau_{\alpha, \beta}\left(k+v_{j}, x\right)
		\end{array}\right.
		$$
		with $\sigma_{0,0}:=\sigma$ and $\tau_{0,0}:=\tau$.   Here $v_i$ always takes value 1 but operation will act with respect to  $i$-th variable. 
		From Proposition \ref{new3},  we note that $\sigma_{\alpha, \beta}$ and $\tau_{\alpha, \beta}$ are the  symbols of  pseudo-differential operators  on $\mathbb{ Z}^n$ and it satisfies 
		$$
		D_{x}^{(\beta)}  {\Delta}_{k}^{\alpha}\frac{\sigma(k, x)}{\tau(k, x)}=\frac{\sigma_{\alpha, \beta}(k, x)}{\tau_{\alpha, \beta}(k, x)}.
		$$
		Therefore by induction on both $\alpha$ and $\beta$, we obtain 
		$$
		\left\{\begin{array}{l}
			|\sigma_{\alpha, \beta}(k, x) |\leq C_{1} \Lambda(k)^{(2^{|\alpha+\beta|}-1) m_2+m_1-\rho|\alpha|} \\
			C_{2} \Lambda(k)^{2^{|\alpha+\beta| }m_2} \leq\left|\tau_{\alpha, \beta}(k, x)\right| \leq C_{3} \Lambda(k)^{2^{|\alpha+\beta|}m_2},
		\end{array}\right.
		$$
		where $C_{1}, C_{2}$ and $C_{3}$ are appropriate constants depends on $\alpha$ and $\beta$. Therefore from the relation  (\ref{s14}),  we get 
		\begin{align*}
			\left|D_{x}^{(\beta)} \Delta_{k}^{\alpha}  \left(k^{\gamma} \Delta_{k}^{\gamma}\left(\frac{\sigma}{\tau}\right)\right)(k, x) \right| &\leq \frac{C_{1} C_{\alpha, \gamma}}{C_{3} }\sum_{\eta \leq \alpha} \left( \begin{array}{l}
				\alpha \\
				\eta
			\end{array}\right) \Lambda(k)^{(2^{|\alpha-\eta+\gamma+\beta|}-1) m_2+m_1-\rho|\alpha-\eta+\gamma|}\\&\qquad   
			\times  \Lambda(k)^{-(2^{|\alpha-\eta+\gamma+\beta|}) m_2}\Lambda(k)^{\rho|\gamma-\eta|} \\
			%	&=\frac{C_{1} C_{\alpha, \gamma}}{C_{3} }\sum_{\eta \leq \alpha} \left( \begin{array}{l}c	\alpha \\c	\eta	\end{array}\right) \Lambda(k)^{m_1-m_2-\rho|\alpha-\eta+\gamma|+\rho|\gamma-\eta|} \\
			&\leq  C_{\alpha, \beta, \gamma}\Lambda(k)^{m_1-m_2-\rho|\alpha|},
		\end{align*}
		where $ C_{\alpha, \beta, \gamma}$ is a constant depends only on $ \alpha, \beta$ and $ \gamma$. This completes the proof of the theorem.

		%	Proof follows from Proposition 3.6 of \cite{kal} and Theorem 4.9.4 of  \cite{ruz1}, since the symbol classes are the same modulo swapping the order of the variable.
	\end{proof}
	\begin{theorem}\label{5}
		Let  $\sigma \in M_{\rho, \Lambda}^m(\mathbb{Z}^n\times \mathbb{T}^n)$ be  $M$-elliptic. Then 
		there exist a symbol $\tau \in M_{\rho, \Lambda}^{-m}(\mathbb{Z}^n\times \mathbb{T}^n)$ such that $$T_\tau T_\sigma=I+R$$ and $$ T_\sigma T_\tau=I+S,$$ where $R$ and $S$ are  pseudo differential operators  with symbols in $  M_{\rho, \Lambda}^{-\infty}\left( \mathbb{Z}^{n} \times \mathbb{T}^{n} \right)$. %$\cap_{k\in\mathbb{R}} M_{\rho, \Lambda}^{k}\left(\mathbb{Z}^{n} \times \mathbb{T}^{n}\right) $.
	\end{theorem}
	\begin{proof}
		%We prove this theorem by the same technique given in Theorem 3.6 \cite{bot}.
		Let $T_\sigma$ is the pseudo-differential operator with  symbol $\sigma\in M_{\rho, \Lambda}^m(\mathbb{Z}^n\times \mathbb{T}^n).$
		Consider the symbol  $\sigma_0(k, x)=\frac{1}{\sigma(k, x)}, (k, x)\in \mathbb{Z}^n\times \mathbb{T}^n$ and let $T_{\sigma_0}$ be the corrrespoing pseudo-differential operators.	Then by the Proposition \ref{3},  the symbol $ \sigma_0(k, x)\in M_{\rho, \Lambda}^{-m}(\mathbb{Z}^n\times \mathbb{T}^n)$ and   $\sigma_0\sigma\sim 1\in M_{\rho,\Lambda}^{0}(\mathbb{Z}^n\times \mathbb{T}^n)$.    Thus,  from the composition formula given in  Theorem \ref{4}, it is true that 
		$$\lambda=\sigma_0\sigma-\sigma_R\sim 1-\sigma_R$$ 
		for some symbol $\sigma_R \in M_{\rho,  \Lambda}^{-1}(\mathbb{Z}^n\times \mathbb{T}^n),$
		equivalently $T_{\sigma_0} T_\sigma\sim I-R$. Note that $$\sum_{j=0}^{N-1} R^j(I-R)=1-R^N,$$ where $R^N\in \mathrm{OP}(M_{\rho, \Lambda}^{-N}(\mathbb{Z}^n\times \mathbb{T}^n))$, so that $S\sim \sum_{j=0}^{\infty}R^j$ is a formal Neumann series of the inverse of $I-R.$ Now 
		$$ST_{\sigma_0} T_\sigma \sim S(I-R)\sim I,$$ so that $ST_{\sigma_0} $ is the candidate for a parametrix $T_\tau.$ Indeed, this technique also produce $T_\tau' \in \mathrm{OP}(M_{\rho, \Lambda}^{-m}(\mathbb{Z}^n\times \mathbb{T}^n))$ satisfying $ T_{\sigma}T_\tau' \sim I.$ Finally, $$T_\tau=T_{\tau} (T_\sigma T_\tau') =(T_{\tau} T_\sigma) T_\tau'=T_\tau'$$  and completes the proof.
	\end{proof}
	The operator $T_\tau $ in the previous theorem is known as a  parametrix  of the $M$-elliptic pseudo-differential operator $T_\sigma.$ From the  above it can be seen immediately  that the principal symbol of a parametrix of an elliptic  pseudo-differential operator $T_{\sigma}$ is given by the symbol $\sigma_{0}=1 / \sigma$ (provided that $\sigma$ does not vanish anywhere, which we may assume).
	
	%Using Theorem 3.2 of \cite{vish} and Theorem 3.3 of \cite{kal},  we have the following result.
	%	\begin{theorem} 	Let $\sigma \in M_{\rho, \Lambda}^{0}\left(\mathbb{Z}^{n} \times \mathbb{T}^{n}\right) .$ Then $T_{\sigma}: \ell^{2}\left(\mathbb{Z}^{n}\right) \rightarrow \ell^{2}\left(\mathbb{Z}^{n}\right)$ is a bounded linear 		operator. 	\end{theorem}
	
	From Theorem 4.1 of  \cite{bot} we have the following    relation between lattice and toroidal quantizations.
	\begin{theorem}\label{l^2}
		Let $\sigma: \mathbb{Z}^{n} \times \mathbb{T}^{n} \rightarrow \mathbb{C}$ be a measurable function such that the pseudo-differential operator $T_{\sigma}: \ell^{2}\left(\mathbb{Z}^{n}\right) \rightarrow \ell^{2}\left(\mathbb{Z}^{n}\right)$ is a bounded linear operator. If we define
		$\tau: \mathbb{T}^{n} \times \mathbb{Z}^{n} \rightarrow \mathbb{C}$ by
		$$
		\tau(x, k)=\overline{\sigma(-k, x)}, \quad k \in \mathbb{Z}^{n}, x \in \mathbb{T}^{n}
		$$
		then $T_{\sigma}=\mathcal{F}_{\mathbb{Z}^{n}} T_{\tau}^{*} \mathcal{F}_{\mathbb{Z}^{n}}^{-1}$, where $T_{\tau}^{*}$ is the adjoint of $T_{\tau} .$
	\end{theorem} 
	In the following, we prove  $\ell^2$-boundedness    of the pseudo-differential operator $T_\sigma$ when  the symbol $\sigma\in M_{\rho, \Lambda }^{0}\left(\mathbb{Z}^{n} \times \mathbb{T}^{n}\right) $.
	\begin{theorem}	\label{eq20}
		Let \(\sigma \in M_{\rho, \Lambda }^{0}\left(\mathbb{Z}^{n} \times \mathbb{T}^{n}\right) .\) Then \(T_{\sigma}: \ell^{2}\left(\mathbb{Z}^{n}\right) \rightarrow \ell^{2}\left(\mathbb{Z}^{n}\right)\) is a bounded linear operator.	\end{theorem} 
	\begin{proof}
		Since  the Fourier transform $\mathcal{F}_{\mathbb{Z}^{n}}: \ell^{2}\left(\mathbb{Z}^{n}\right) \rightarrow L^{2}\left(\mathbb{T}^{n}\right)$
		is an isometry, using  Theorem \ref{l^2}, it follows that $T_{\sigma}: \ell^{2}\left(\mathbb{Z}^{n}\right) \rightarrow \ell^{2}\left(\mathbb{Z}^{n}\right)$ is a bounded linear operator if and
		only if $T_{\tau}: L^{2}\left(\mathbb{T}^{n}\right) \rightarrow L^{2}\left(\mathbb{T}^{n}\right)$ is a bounded linear operator.  Also we note that $\sigma \in M_{\rho, \Lambda }^{0}\left(\mathbb{Z}^{n} \times \mathbb{T}^{n}\right) $
		if and only if $\tau \in M^{0}_{\rho, \Lambda }\left(\mathbb{T}^{n} \times \mathbb{Z}^{n}\right) .$  But in Theorem 3.3 of \cite{kal}, Kalleji proved that   if $\tau \in M^{0}_{\rho, \Lambda }\left(\mathbb{T}^{n} \times \mathbb{Z}^{n}\right) $  then $ T_{\tau}: L^{2}\left(\mathbb{T}^{n}\right) \rightarrow L^{2}\left(\mathbb{T}^{n}\right)$ is a
		bounded linear operator. Thus if \(\sigma \in M_{\rho, \Lambda }^{0}\left(\mathbb{Z}^{n} \times \mathbb{T}^{n}\right) \),  then \(T_{\sigma}: \ell^{2}\left(\mathbb{Z}^{n}\right) \rightarrow \ell^{2}\left(\mathbb{Z}^{n}\right)\) is a bounded linear operator. This completes the proof.
	\end{proof}

	\section{Minimal maximal operators}	\label{sec4}

	In this  section,  we  define weighted Sobolev space and  investigate  the minimal and maximal pseudo-differential operators on $\ell^2(\mathbb{Z}^n)$.    We show that show that they coincide on $\ell^2(\mathbb{Z}^n)$ subject to the  $M$-ellipticity of symbols. We also determine the domains of the minimal and maximal operators.%	{\color{red} need to add some more lines about this section }

	For $s\in \mathbb{R}$, the weighted Sobolev space $\mathcal{H}_{ \Lambda}^{s, 2}(\mathbb{Z}^n)$ is defined by	
	$$	\mathcal{H}_{ \Lambda}^{s, 2}(\mathbb{Z}^n)=\left\{w \in \mathcal{S}'(\mathbb{Z}^n): 	\|w\|_{s, 2,\Lambda}<\infty \right\},	$$
	where the norm $\|.\|_{ s, 2,\Lambda}$ is given by	$$	\|w\|_{s, 2,\Lambda}=\left\|\Lambda(D)^{s} w\right\|_{\ell^{2}\left(\mathbb{Z}^n\right)}:=\left(\sum_{k \in \mathbb{Z}^{n}}\Lambda(k)^{2s }|w(k)|^{2}\right)^{\frac{1}{2}}, \quad w \in \mathcal{H}_{\Lambda}^{s, 2}.	$$
	Moreover, $\left( 	\mathcal{H}_{ \Lambda}^{s, 2}(\mathbb{Z}^n), \|.\|_{ s, 2,\Lambda} \right)$  is a Banach space. Recall that the symbol $\sigma_{s}(k)=\Lambda(k)^{s } \in M_{\rho, \Lambda}^{s} \left(\mathbb{Z}^{n}\times\mathbb{T}^n \right)$ and denote by $\Lambda(D)^s$   the corresponding operator instead of $T_ {\sigma_{s}}$ or $\operatorname{OP}(\sigma_s).$   Further,  we have $f \in 	\mathcal{H}_{ \Lambda}^{s, 2}(\mathbb{Z}^n)$ if and only if $\Lambda(D)^sf \in \ell^{2}\left(\mathbb{Z}^{n}\right) .$ Consequently, 	\begin{align}\label{eq30}			\mathcal{H}_{ \Lambda}^{s, 2}(\mathbb{Z}^n)=\Lambda(D)^{-s}\left(\ell^{2}\left(\mathbb{Z}^{n}\right)\right).	\end{align}	
	%	From onwards, when the symbol %the symbol is the weight function	 $\sigma_{s}(k)=\Lambda(k)^{s }$, we denote  the corresponding pseudo-differential operator as $\Lambda(D)^s$ instead of $T_ {\sigma_{s}}$ or $\operatorname{OP}(\sigma_s).$ 	
	Using the above notations, we have the following result. 
	
	%From onwards,  we use the notation $ \Lambda(D)^s$ as the $T_ {a_{s}}$

	%	Using  the duality property  $\mathcal{F}_{\mathbb{Z}^n}^{-1}:   C^{\infty}(\mathbb{T}^n) \to  \mathcal{S}(\mathbb{Z}^n)$, the discrete  Fourier transform can be extended uniquely to the mapping $\mathcal{F}_{\mathbb{Z}^n}: \mathcal{S}'(\mathbb{Z}^n)\to D'(\mathbb{T}^n)$ by the formula 	$$\langle \mathcal{F}_{\mathbb{Z}^n} w, u \rangle =\langle w, p\circ \mathcal{F}_{\mathbb{Z}^n}^{-1}u \rangle,$$	where $w\in \mathcal{S}'(\mathbb{Z}^n), u\in C^{\infty}(\mathbb{T}^n) $ and $p$ is defined by $(p\circ u)(k)=u(-k)$.  Let  $s\in \mathbb{R}.$ Then $\Lambda(D)^s$ is defined to be the Fourier multiplier given by $$\Lambda(D)^s w= \mathcal{F}_{\mathbb{Z}^n}^{-1}\Lambda^s \mathcal{F}_{\mathbb{Z}^n}w, ~~w\in  \mathcal{S}'(\mathbb{Z}^n).$$ 	
	
	%	We follow the notation  $\mathcal{H}_{ \Lambda}^{s, 2}(\mathbb{Z}^n)$ as $ \mathcal{H}_{ \Lambda}^{s, 2}$. 

	\begin{proposition}
		For  $s\in \mathbb{R}, \Lambda(D)^s: 	\mathcal{H}_{ \Lambda}^{s, 2}(\mathbb{Z}^n) \rightarrow \ell^{2}\left(\mathbb{Z}^{n}\right)$ is a surjective isometry.
	\end{proposition} 
	\begin{proof}
		Since 	$\|w\|_{s, 2,\Lambda}=\left\|\Lambda(D)^{s} w\right\|_{\ell^{2}\left(\mathbb{Z}^n\right)}$ for any $w \in \mathcal{H}_{\Lambda}^{s, 2},$ it follows that
		$	\Lambda(D)^s: 	\mathcal{H}_{ \Lambda}^{s, 2}(\mathbb{Z}^n) \rightarrow \ell^{2}\left(\mathbb{Z}^{n}\right)$ is a   isometry. Now for every $v \in\ell^{2}\left(\mathbb{Z}^{n}\right)$, we set $w:=\left(	\Lambda(D)^s\right)^{-1} v .$ Then $	\Lambda(D)^s w=v \in \ell^{2}\left(\mathbb{Z}^{n}\right)$.  Hence, $ \Lambda(D)^s: 	\mathcal{H}_{ \Lambda}^{s, 2}(\mathbb{Z}^n) \rightarrow \ell^{2}\left(\mathbb{Z}^{n}\right)$ is surjection.
	\end{proof}
	The following result know as the Sobolev embedding theorem. 
	\begin{proposition}\label{1}		For   $m_1 \leq m_2, \mathcal{H}_{\Lambda}^{m_2, 2} \subseteq  \mathcal{H}_{ \Lambda}^{m_1, 2}$ and		$$		\|w\|_{ m_1, 2,\Lambda} \leq \|w\|_{m_2, 2,\Lambda}, ~~ w\in \mathcal{H}_{\Lambda}^{m_1, 2}		.$$
	\end{proposition}
	\begin{proof}	Let $w\in \mathcal{H}_{\Lambda}^{m_2, 2}.$  Since $m_1\leq m_2$,  we have 		\begin{align*} 		\|w\|_{m_1, 2,\Lambda} &=\left\|\Lambda(D)^{m_1}   w\right\|_{\ell^{2}\left(\mathbb{Z}^n\right)}\\&=\left\|  \Lambda(D)^{m_1-m_2}\Lambda(D)^{m_2}w\right\|_{\ell^{2}\left(\mathbb{Z}^n\right)}\\& 	\leq  \left\|  \Lambda(D)^{m_2}w\right\|_{\ell^{2}\left(\mathbb{Z}^n\right)}	=\|w\|_{m_2, 2,\Lambda} 	\end{align*}   and consequently   $\mathcal{H}_{\Lambda}^{m_2, 2} \subseteq  \mathcal{H}_{ \Lambda}^{m_1, 2}$.  \end{proof}
	\begin{corollary}\label{2}	
		Let  $\sigma \in M_{\rho, \Lambda}^{m}\left(\mathbb{Z}^{n} \times \mathbb{T}^{n}\right) $. Then the operator $T_\sigma $   extends to a bounded operator from $ \mathcal{H}_{ \Lambda}^{s+m, 2} $  to $ \mathcal{H}_{ \Lambda}^{s, 2}$ for any $s \in \mathbb{R}.$
	\end{corollary}
	\begin{proof}
		Let us consider the   operator $A=\Lambda(D)^{s} T_\sigma \Lambda(D)^{-s-m}$. Then the    symbol of $A$ is in $M_{\rho, \Lambda}^{0}\left(\mathbb{Z}^{n} \times \mathbb{T}^{n}\right) .$   Thus from Theorem \ref{eq20}, the operator $A: \ell^{2}\left(\mathbb{Z}^{n}\right) \to  \ell^{2}\left(\mathbb{Z}^{n}\right)$ is  bounded. Now, for every $w\in \mathcal{H}_{ \Lambda}^{s+m, 2} $, we have 
		\begin{align*}
			\|T_\sigma w\|_{s, 2,\Lambda} &=\left\|\Lambda(D)^{s} T_\sigma w\right\|_{\ell^{2}\left(\mathbb{Z}^n\right)}\\
			&=\left\|A \; \Lambda(D)^{s+m}w\right\|_{\ell^{2}\left(\mathbb{Z}^n\right)}\\
			&\leq C \left\| \Lambda(D)^{s+m}w\right\|_{\ell^{2}\left(\mathbb{Z}^n\right)}	=\|w\|_{ s+m, 2,\Lambda}.
		\end{align*}
		This shows that $T_\sigma $   is a   bounded operator from $ \mathcal{H}_{ \Lambda}^{s+m, 2} $  to $ \mathcal{H}_{ \Lambda}^{s, 2}$ for any $s \in \mathbb{R}.$
	\end{proof}
	
	For $-\infty<m<\infty$, we   define $M_{\rho, \Lambda, 0}^{m}\left(\mathbb{Z}^{n} \times \mathbb{T}^{n}\right) $ to be the set of symbols $\sigma\in M_{\rho, \Lambda}^{0}\left(\mathbb{Z}^{n} \times \mathbb{T}^{n}\right) $ such that for all multi-indices $\alpha, \beta$  and $ \gamma \in\{0,1\}^n$, there exists a bounded real-valued function $C_{\alpha, \beta}$ on $\mathbb{Z}^{n}$ for which 
	$$	\left| k^\gamma D_{x}^{(\beta)} \Delta_{k}^{\alpha+\gamma} \sigma(k, x)\right| \leq C_{ \alpha, \beta,\gamma}\Lambda(k)^{m-\rho|\alpha|}$$
	and
	$$
	\lim _{|k| \rightarrow \infty} C_{\alpha, \beta,\gamma}(k)=\lim _{|k| \rightarrow \infty} C_{\alpha, \beta, \gamma}\Lambda(k)^{m-|\alpha|}=0.
	$$
	In particular, we have  $M_{\rho, \Lambda}^{-m}\left(\mathbb{Z}^{n} \times \mathbb{T}^{n}\right) \subseteq M_{\rho, \Lambda, 0}^{m}\left(\mathbb{Z}^{n} \times \mathbb{T}^{n}\right) $. Then the following discrete weighted $\ell^{2}$-version of for discrete pseudo-differential operators can be proved using suitable 
	modifications given in Theorem 3.2 in  \cite{wongggg}.
	\begin{theorem}\label{eq12}
		Let $\sigma \in M_{\rho, \Lambda, 0}^{m}\left(\mathbb{Z}^{n} \times \mathbb{T}^{n}\right)$  and $s\in \mathbb{R}$. Then for every positive number $\varepsilon, T_{\sigma}: \mathcal{H}_{ \Lambda}^{s+m, 2} \left(\mathbb{Z}^{n}\right) \rightarrow \mathcal{H}_{ \Lambda}^{s-\varepsilon, 2} \left(\mathbb{Z}^{n}\right)$ is a compact operator.
	\end{theorem}
	\begin{proof}
		Let $\phi \in C_{0}^{\infty}\left(\mathbb{R}^{n}\right)$ be such that $\phi(x)=1$ for $|x| \leq 1$. For any positive integer $k_1$, let $\sigma_{k_1}$ be the function on $\mathbb{Z}^{n} \times \mathbb{T}^{n}$ defined by
		$$
		\sigma_{k_1}(k, x)=\phi\left(\frac{k}{k_1}\right) \sigma(k, x).
		$$
		Then $T_{\sigma_{k_1}}=\phi_{k_1} T_{\sigma}$, where
		$$
		\phi_{k_1}(k)=\phi\left(\frac{k}{k_1}\right), \quad k \in \mathbb{Z}^{n}.
		$$
		Now $T_{\sigma_{k_1}}$ is compact because $T_{\sigma}$ is a bounded linear operator from   $\mathcal{H}_{ \Lambda}^{s+m, 2} $  into $ \mathcal{H}_{ \Lambda}^{s, 2}$ by Corollary  \ref{2} and the multiplication operator $\phi_{k_1}$ from $ \mathcal{H}_{ \Lambda}^{s, 2} $  into $ \mathcal{H}_{ \Lambda}^{s-\varepsilon, 2}$  is a compact operator by the Sobolev embedding theorem.  Since $T_{\sigma_{k_1}} \rightarrow T_{\sigma}$ as $k_1 \rightarrow \infty$, it implies that $T_{\sigma}$ is compact.
	\end{proof}
	An immediate    consequence of Theorem \ref{eq12} is the following corollary.
	\begin{corollary}\label{ref13}
		For every positive number $\varepsilon, \Lambda(D)^{-\varepsilon}: \ell^{2}\left(\mathbb{Z}^{n}\right) \rightarrow \ell^{2}\left(\mathbb{Z}^{n}\right)$ is a compact operator.
	\end{corollary} 
	
	\begin{theorem} \label{eq17}
		Let $m_1, m_2\in \mathbb{R}$	 such that  $m_1 \leq m_2$. Then  the inclusion map $i: \mathcal{H}_{\Lambda}^{m_2, 2} \to  \mathcal{H}_{ \Lambda}^{m_1, 2}$ is  a compact operator.
	\end{theorem}
	\begin{proof}
		Let $\epsilon>0$ be such that $ m_2-m_1-\epsilon>0 $. 
		Let us consider    the   operator $A= \Lambda(D)^{m_1} \Lambda(D)^{\varepsilon}.$ Then $A$ is a discrete pseudo-differential operator of order $m_1+\epsilon$.  Since $A:  \mathcal{H}_{\Lambda}^{m_2, 2} \to  \mathcal{H}_{ \Lambda}^{m_2-m_1-\varepsilon, 2}$ is a bounded linear operator by Theorem \ref{eq12} and $\Lambda(D)^{-\varepsilon}: \ell^{2}\left(\mathbb{Z}^{n}\right) \rightarrow \ell^{2}\left(\mathbb{Z}^{n}\right)$ is a compact operator by Corollary \ref{ref13}, it  follows that the composition $\Lambda(D)^{-\varepsilon} i A$ of the mappings
		\begin{align*}
			&A:  \mathcal{H}_{\Lambda}^{m_2, 2} \to  \mathcal{H}_{ \Lambda}^{m_2-m_1-\varepsilon, 2} \\
			&	i:\mathcal{H}_{ \Lambda}^{m_2-m_1-\varepsilon, 2}  \hookrightarrow \ell^{2}\left(\mathbb{Z}^{n}\right) 
		\end{align*}
		and
		$$
		\Lambda(D)^{-\varepsilon}: \ell^{2}\left(\mathbb{Z}^{n}\right) \rightarrow \ell^{2}\left(\mathbb{Z}^{n}\right) 
		$$
		is a compact operator. Thus the linear operator 
		$$ \mathcal{H}_{\Lambda}^{m_2, 2} \ni u \to  \Lambda(D)^{-\varepsilon} i A u= \Lambda(D)^{\varepsilon}u\in \ell^{2}\left(\mathbb{Z}^{n}\right) $$ is  a compact.
	\end{proof}

	%	\begin{theorem} 		Let $m_1, m_2 \in \mathbb{R}$. 	Let $\sigma  \in M_{\rho, \Lambda}^{m_1}\left(\mathbb{Z}^{n} \times \mathbb{T}^{n}\right) $ and $\tau  \in M_{\rho,  \Lambda}^{m_2}\left(\mathbb{Z}^{n} \times \mathbb{T}^{n}\right) $. Then the product  $T_\sigma T_\tau$ of the pseudo-differential operators $T_\sigma$  and $T_\tau $ is  a pseudo-differential operator  with symbol $\lambda\in M_{\rho,  \Lambda}^{m_1+m_2}\left(\mathbb{Z}^{n} \times \mathbb{T}^{n}\right) $ having  asymptotic expansion	\begin{align*}\lambda(k, x)& {\sim \sum_{\alpha \geq 0} \frac{1}{\alpha !}\left(D_{x}^{(\alpha)} \sigma (k, x)\right) \Delta_k^{\alpha} \tau(k, x)},	\end{align*}	where the asymptotic expansion means that for every $N \in \mathbb{N}$ we have $$\lambda(k, x)-\sum_{|\alpha|<N} \frac{1}{\alpha !}\left(D_{x}^{(\alpha)} \sigma (k, x)\right) \Delta_k^{\alpha} \tau(k, x)\in \Lambda(\xi) ^{m_{1}+m_{2}-\rho N}. $$ \end{theorem} 

	\begin{definition}		A operator $T$ from a Banach space $X$ to another  Banach space $Y$ is said to be closable if and only if for any sequence $\{  x_k\} $ in $Dom(T)$  such that $x_k \to 0 $ in $X$ and $T(x_k) \to y $ as $k\to \infty $, then $y=0.$	\end{definition}
	\begin{proposition} Let $\sigma \in M_{\rho,  \Lambda}^{m}(\mathbb{Z}^n\times \mathbb{T}^n) $. The linear operator $T_{\sigma}: \ell^{2}(\mathbb{Z}^n) \rightarrow \ell^{2}(\mathbb{Z}^n)$ is closable with dense domain  $\mathcal{S}(\mathbb{Z}^n)$ in $ \ell^{2}(\mathbb{Z}^n)$.
	\end{proposition}
	\begin{proof}	Let $\left(\phi_{k}\right)_{k \in \mathbb{N}}$ be a sequence in $\mathcal{S}(\mathbb{Z}^n)$ such that $\phi_{k} \rightarrow 0$ and $T_{\sigma} \phi_{k} \rightarrow f$ as $k \rightarrow \infty$  for some $f$ in	$\ell^{2}(\mathbb{Z}^n)$.  We have to show that $f=0$. 		For any $\psi\in \mathcal{S}(\mathbb{Z}^n)$, we have 	
		\begin{align*}		\langle T_{\sigma} \phi_{k}, \psi\rangle =\langle \phi_{k}, T_{\sigma}^*  \psi\rangle,		\end{align*}	
		where $T_{\sigma}^* $ is the adjoint of $T_{\sigma}$.	Taking limit as $k \rightarrow \infty$, we have $\langle f, \psi\rangle =0$ for all $\psi \in  \mathcal{S}(\mathbb{Z}^n)$. Since  $ \mathcal{S}(\mathbb{Z}^n)$ is dense in   $ \ell^{2}(\mathbb{Z}^n)$, it follows that $f=0.$ Hence $T_{\sigma}: \ell^{2}(\mathbb{Z}^n) \rightarrow \ell^{2}(\mathbb{Z}^n)$ is closable.	\end{proof}

	For $\sigma \in M_{\rho,  \Lambda}^{m}(\mathbb{Z}^n\times \mathbb{T}^n)$, consider the pseudo differential operator $T_{\sigma}: \ell^{2}(\mathbb{Z}^n) \rightarrow \ell^{2}(\mathbb{Z}^n)$ with domain $\mathcal{S}(\mathbb{Z}^n)$. Then by the previous proposition  $T_{\sigma}$  has a closed extension. Let $T_{\sigma, 0}$ be the minimal operator for $T_{\sigma}$. Let us recall, the domain ${Dom}\left(T_{\sigma, 0}\right)$ of $T_{\sigma, 0}$ consists of all functions $g \in \ell^{2}(\mathbb{Z}^n)$ for which there exists a sequence $\left(\phi_{k}\right)_{k \in \mathbb{N}}$ in $\mathcal{S}(\mathbb{Z}^n)$ such that $\phi_{k} \rightarrow g$ in $\ell^{2}(\mathbb{Z}^n)$ and $T_{\sigma} \phi_{k} \rightarrow f$	for some $f \in \ell^{2}(\mathbb{Z}^n)$ as $k \rightarrow \infty$.  Moreover, $T_{\sigma, 0}g=f$.	Further,  it can be shown that \(f\) does not depend on the choice of \(\left(\phi_{k}\right)_{k \in \mathbb{N}}\) and \(T_{\sigma, 0} g=f .\)	
	\begin{definition}\label{8}	Let $T_{\sigma, 1}$  be a  linear operator on $\ell^{2}(\mathbb{Z}^n)$. Let $f$ and $g$ be  two function in $\ell^{2}(\mathbb{Z}^n)$. Then we say that $g \in {Dom}\left(T_{\sigma, 1}\right)$ and $T_{\sigma, 1} g=f $	if and only if	
		\begin{align}\label{7}	
			\langle g, T_{\sigma}^{*} \psi\rangle =\langle f, \psi \rangle,~~\quad\forall \psi \in \mathcal{S}(\mathbb{Z}^n)
		\end{align}	
		where $T_{\sigma}^{*}$ is the adjoint of $T_{\sigma}.$	\end{definition}	
	\begin{proposition}\label{23}	$T_{\sigma, 1}$ is a closed linear operator from $\ell^{2}(\mathbb{Z}^n)$ into $\ell^{2}(\mathbb{Z}^n)$ with domain ${Dom}\left(T_{\sigma, 1}\right)$ containing $\mathcal{S}(\mathbb{Z}^n).$  Moreover, $T_{\sigma,1} u=T_{\sigma}u$  for $u\in {Dom}\left(T_{\sigma, 1}\right)$ in distribution sense.	
	\end{proposition}
	\begin{proof}		From the above definition, clearly $\mathcal{S}(\mathbb{Z}^n) \subset {Dom}\left(T_{\sigma, 1}\right)$. 	Let $(\phi_{j})_{j\in \mathbb{N}}$ be a sequence  in ${Dom}\left(T_{\sigma, 1}\right)$ such that $\phi_{j} \rightarrow u$ and $T_{\sigma, 1} \phi_{j} \rightarrow f$ in \(\ell^{2}\left(\mathbb{Z}^{n}\right)\) for	some \(u\) and \(f\) in \(\ell^{2}\left(\mathbb{Z}^{n}\right)\) as \(j \rightarrow \infty\). Then by the relation  (\ref{7}), we have	$$\langle\phi_{j}, T_{\sigma}^{*} \psi\rangle=\langle T_{\sigma, 1} \phi_{j}, \psi\rangle, ~~\quad \psi \in \mathcal{S}\left(\mathbb{Z}^{n}\right).$$ Letting \(j \rightarrow \infty\)  we have $	\langle u, T_{\sigma}^{*} \psi\rangle=\langle  f, \psi\rangle$ for all $\psi \in \mathcal{S} (\mathbb{Z}^{n} ).$ By definition \ref{8}, we have  $u\in {Dom}\left(T_{\sigma, 1}\right)$   and  $T_{\sigma, 1}u=f.$   Hence $T_{\sigma, 1}$ is a closed operator from $\ell^{2}(\mathbb{Z}^n)$ into $\ell^{2}(\mathbb{Z}^n)$ with domain ${Dom}\left(T_{\sigma, 1}\right)$ containing $\mathcal{S}(\mathbb{Z}^n).$ It can be prove easily that for $u\in {Dom}\left(T_{\sigma, 1}\right)$, $T_{\sigma,1} u=T_{\sigma}u$.	\end{proof}	
	\begin{proposition}\label{13}	
		$T_{\sigma, 1}$ is an extension of $T_{\sigma, 0}.$ 	
	\end{proposition}
	\begin{proof}		Let \(u \in {Dom}\left(T_{\sigma, 0}\right)\) and \(T_{\sigma, 0} u=f .\) Then  by the definition of  ${Dom}\left(T_{\sigma, 0}\right)$, there exists a sequence $(\phi_{j})_{j\in \mathbb{N}}$ in		\(\mathcal{S}(\mathbb{Z}^n)\) such that \(\phi_{j} \rightarrow u\) and \(T_{\sigma} \phi_{j} \rightarrow f\) in \(\ell^{2}\left(\mathbb{Z}^{n}\right)\) as \(j \rightarrow \infty \) for some $f\in \ell^2 (\mathbb{Z}^n)$. Since $	\langle \phi_{j}, T_{\sigma}^{*}\psi \rangle=\langle T_{\sigma} \phi_{j}, \psi\rangle$		for all $\psi \in \mathcal{S}(\mathbb{Z}^n)$,  letting \(j \rightarrow \infty\), we have $$		\langle u, T_{\sigma}^{*} \psi\rangle=\langle  f, \psi\rangle, \quad \text{for all} ~\psi \in \mathcal{S}\left(\mathbb{Z}^{n}\right).$$ Thus,  from  the definition \ref{8}, we get $ u\in {Dom}\left(T_{\sigma, 1}\right)$ and $T_{\sigma,1} u=f$. Hence $T_{\sigma, 1}$ is an extension of $T_{\sigma, 0}.$ 	\end{proof}	
	\begin{proposition}		$\mathcal{S}(\mathbb{Z}^n) \subseteq  Dom(T_{\sigma,1} ^ { t } )$, where $T_{\sigma,1} ^ { t } $ is the transpose of $T_{\sigma, 1}$.	\end{proposition}	
	\begin{proposition}	$T_{\sigma, 1}$ is the largest closed extension of $T_{\sigma}$ in the sense that if $S$ is any closed extension of $T_{\sigma}$ such that $\mathcal{S}(\mathbb{Z}^n) \subseteq {Dom}  (S^{t} ),$ then $T_{\sigma, 1}$ is an extension of $S$. Such $T_{\sigma, 1}$ is called the maximal operator of $T_{\sigma}.$		\end{proposition}	
	\begin{corollary}	
		In view of the above we call $T_{\sigma, 0} $ and $T_{\sigma, 1}$  are the minimal and maximal pseudo differential operator of $T_\sigma.$
	\end{corollary}	
	
	\begin{proposition}\label{11}	Let  $\sigma \in M_{\rho, \Lambda}^m(\mathbb{Z}^n\times \mathbb{T}^n)$ be  $M$-elliptic.   Then there exist positive constants $C$ and $D>0$ such that		$$		C\|u\|_{m, 2,  \Lambda} \leq \|T_\sigma u\|_{0,  2,  \Lambda} +\|u\|_{0, 2, \Lambda} \leq D\|u\|_{m,  2,  \Lambda} 		$$		for every $u\in {H}_{  \Lambda}^{m, 2}(\mathbb{Z}^n ).$	\end{proposition}
	\begin{proof} 	
		From  Theorem \ref{5},	for every $u\in \mathcal{H}_{\Lambda}^{m, 2}(\mathbb{Z}^n ),$ we have 	$$	u=T_\tau T_\sigma u-R u	$$	 where $\tau \in M_{\rho, \Lambda}^{-m}(\mathbb{Z}^{n} \times \mathbb{T}^{n})$ and  $R$  is a  pseudo differential operators  with symbols in $M_{\rho, \Lambda}^{-\infty}(\mathbb{Z}^n\times \mathbb{T}^n) $. Thus  using  Corollary \ref{2},  we get
		\begin{align*}
			\|u\|_{m, 2,  \Lambda}&= \|T_\tau (T_\sigma u)-R u\|_{m, 2, \Lambda}\\	
			&\leq \|T_\tau (T_\sigma u)\|_{m, 2,  \Lambda} +\|Ru\|_{m, 2,  \Lambda}\\	
			&\leq C_1\|T_\sigma u\|_{0,  2,  \Lambda}+C_2\|u\|_{0,  2,  \Lambda}\\	&\leq \max\{ C_1, C_2\}\left(\|T_\sigma u\|_{0,  2,  \Lambda}+\|u\|_{0,  2,  \Lambda}\right)	\end{align*}	
		for every $u\in \mathcal{H}_{ \Lambda}^{m, 2}(\mathbb{Z}^n).$  This proves the left hand side inequality. 	For the other side,  from   Remark \ref{1} and Corollary \ref{2}, there exists a   constants $D'$ such that 
		\begin{align*}	\|T_\sigma u\|_{0,  2,  \Lambda} +\|u\|_{0,  2,  \Lambda} &\leq D'\|u\|_{m,  2,  \Lambda} +\|u\|_{m,  2,  \Lambda} \\		&=(D'+1) \|u\|_{m,  2,  \Lambda}=D\|u\|_{m,  2,  \Lambda}
		\end{align*}	
		for every $u\in \mathcal{H}_{\Lambda}^{m, 2}(\mathbb{Z}^n).$	This completes the proof of the Proposition. %The inequality $$ 		C\|g\|_{m, 2, \Lambda } \leq\left\|T_{\sigma} g\right\|_{0, 2, \Lambda }+\|g\|_{0, 2, \Lambda }  	$$ 	follows from the  inequality \ref{9} $$ \|u\|_{ s+\mu, 2, \Lambda} \leq C_{s N}\left(\|f\|_{ s, 2, \Lambda}+\|u\|_{ -N,2, \Lambda}\right).$$ %	with $s=0, N=-2$ (we know $\mathcal{H}_{\mathfrak{L}, \Lambda}^{0}(\Omega)=L^{2}(\Omega)$). The other  inequality $$	\left\|T_{\sigma} g\right\|_{L^{2}(\Omega)}+\|g\|_{L^{2}(\Omega)} \leq D\|g\|_{\mathcal{H}_{\mathfrak{L}, \Lambda}^{m}(\Omega)}$$	is follows from  the boundedness of $ T_{\sigma}: \mathcal{H}_{\mathfrak{L}, \Lambda}^{m}(\Omega) \rightarrow \mathcal{H}_{\mathfrak{L}, \Lambda}^{0}(\Omega)=L^{2}(\Omega)$ given in Corollary \ref{5}.	
	\end{proof}
	\begin{proposition}\label{10} 	$\mathcal{S}(\mathbb{Z}^n)$ is dense  in $\mathcal{H}_{ \Lambda}^{m, 2}(\mathbb{Z}^n)$.
	\end{proposition}
	\begin{proof}	
		Let $g\in \mathcal{H}_{ \Lambda}^{m, 2}(\mathbb{Z}^n).$ Then from the definition, we have $\Lambda(D)^m g\in \ell^2 (\mathbb{Z}^n).$ Density of  $\mathcal{S}(\mathbb{Z}^n)$  in $ \ell^2(\mathbb{Z}^n)$ implies there exist a sequence $(g_k)_{k\in \mathbb{N}} $ in $ \mathcal{S}(\mathbb{Z}^n)$  such that $ g_k\to\Lambda(D)^m g $ in $\ell^2(\mathbb{Z}^n).$ Let us define $h_k=\Lambda(D)^{-m}g_k.$ Then $h_k\in \mathcal{S}(\mathbb{Z}^n)$ for $k=1, 2, \cdots,$ and
		\begin{align*}	\|h_k-g\|_{m,  2,  \Lambda}&=\|\Lambda(D)^{m}(h_k-g)\|_{\ell^2(\mathbb{Z}^n)}\\	&=\|\Lambda(D)^{m}h_k-\Lambda(D)^{m}g\|_{\ell^2(\mathbb{Z}^n)}\\	&=\|g_k-\Lambda(D)^{m}g\|_{\ell^2(\mathbb{Z}^n)}\to 		\end{align*}
		as $k\to \infty.$ Therefore $\mathcal{S}(\mathbb{Z}^n)$ is dense  in $\mathcal{H}_{ \Lambda}^{m, 2}(\mathbb{Z}^n)$.	
	\end{proof}	
	\begin{theorem}	\label{19}
		Let  $\sigma \in M_{\rho, \Lambda}^m(\mathbb{Z}^n\times \mathbb{T}^n)$ be  $M$-elliptic.    Then $$ \mathcal{H}_{\Lambda}^{m, 2}(\mathbb{Z}^n) = Dom( T_{\sigma, 0}) .$$	
	\end{theorem}
	\begin{proof}	Let $u \in \mathcal{H}_{ \Lambda}^{m, 2}(\mathbb{Z}^n).$  Since $\mathcal{S}(\mathbb{Z}^n)$ is dense  in $\mathcal{H}_{ \Lambda}^{m, 2}(\mathbb{Z}^n)$, there exists a sequence $\left(u_{k}\right)_{k \in \mathbb{N}}$ in $\mathcal{S}(\mathbb{Z}^n)$ such that $u_{k} \rightarrow u$ in $\mathcal{H}_{ \Lambda }^{m, 2}(\mathbb{Z}^n)$  and therefore in $\ell^{2}(\mathbb{Z}^n)$		as $k \rightarrow \infty$. Using Proposition \ref{11}, the sequences $\left(u_{k}\right)_{k \in \mathbb{N}}$ and $\left(T_{\sigma} u_{k}\right)_{k \in \mathbb{N}}$ are Cauchy sequences in $\ell^{2}(\mathbb{Z}^n)$.  	Hence $u_{k} \rightarrow u$	and $T_{\sigma} u_{k} \rightarrow f$ for some $f \in \ell^{2}(\mathbb{Z}^n)$ as $k \rightarrow \infty$. By the definition of $T_{\sigma, 0}$, this implies that $u \in {Dom}\left(T_{\sigma, 0}\right)$ and $T_{\sigma, 0} g=f$. Therefore 
		$\mathcal{H}_{\Lambda}^{m, 2}(\mathbb{Z}^n)\subseteq {Dom}\left(T_{\sigma, 0}\right).	$
		
		Again, let us assume that $u\in  {Dom}(T_{\sigma, 0}).$ Then  by the definition of  ${Dom}\left(T_{\sigma, 0}\right)$, there exists a sequence $( \phi_{j} )_{j\in \mathbb{N}}$ in		\(\mathcal{S}(\mathbb{Z}^n)\) such that \(\phi_{j} \rightarrow u\) and \(T_{\sigma} \phi_{j} \rightarrow f\) in \(\ell^{2}\left(\mathbb{Z}^{n}\right)\) as \(j \rightarrow \infty \) for some $f\in \ell^2 (\mathbb{Z}^n)$.  Then by Proposition \ref{11},  $\left(\phi_{k}\right)_{k \in \mathbb{N}}$ is a Cauchy sequences in $\mathcal{H}_{\Lambda}^{m, 2}(\mathbb{Z}^n)$. Since $\mathcal{H}_{\Lambda}^{m, 2}(\mathbb{Z}^n)$ is complete, there exist  $g\in \mathcal{H}_{\Lambda}^{m, 2}(\mathbb{Z}^n)$ such that $\phi_{k}\to g$ in $\mathcal{H}_{\Lambda}^{m, 2}(\mathbb{Z}^n)$ as  $k\to \infty.$  This implies   $\phi_{k}\to g$ in $\ell^2(\mathbb{Z}^n)$  and hence  $f=g\in \mathcal{H}_{\Lambda}^{m, 2}(\mathbb{Z}^n)$. Therefore ${Dom}\left(T_{\sigma, 0}\right) \subseteq  	\mathcal{H}_{\Lambda}^{m, 2}(\mathbb{Z}^n)	$.

	\end{proof}

	\begin{theorem}
		Let   $\sigma \in M_{\rho, \Lambda}^{m}(\mathbb{Z}^n  \times \mathbb{T}^n )$  be  $M$-elliptic.  Then 	$T_{\sigma, 0} =T_{\sigma, 1} $.
	\end{theorem}
	\begin{proof}
		Since $T_{\sigma, 1}$ is an closed extension of $T_{\sigma, 0} $,  by Theorem \ref{19}, it is enough to show that ${Dom}\left(T_{\sigma, 1}\right) \subseteq \mathcal{H}_{\Lambda}^{m}(\mathbb{Z}^n)$.  Let $u\in {Dom}\left(T_{\sigma, 1}\right) .$ Since $\sigma$ is an $M$-elliptic, by Theorem \ref{5}, there exists $\tau \in M_{\rho, \Lambda}^{  -m}(\mathbb{Z}^n  \times \mathbb{T}^n )$	such that	
		\begin{align}\label{113}		u=T_{\tau} T_{\sigma} u-Ru		\end{align}		
		for every $u\in {M}_{\Lambda}^{m, 2}(\mathbb{Z}^n ),$ where $R$  is a  pseudo differential operators  with symbols in $M_{\rho, \Lambda}^{-\infty}(\mathbb{Z}^n\times \mathbb{T}^n)$.  From Proposition \ref{23} we know that  $T_{\sigma} u=T_{\sigma, 1} u $ in the sense of  distribution. Thus   it follows from 	Corollary  \ref{2} that $T_{\tau} T_{\sigma} u\in \mathcal{H}_{ \Lambda}^{m, 2}(\mathbb{Z}^n)$.  Since $g\in \ell^2(\mathbb{Z}^n)$ and $R$ has symbol in $M_{\rho, \Lambda}^{-\infty} (\mathbb{Z}^n  \times \mathbb{T}^n )\subset M_{\rho,  \Lambda}^{-m}(\mathbb{Z}^n  \times \mathbb{T}^n )$, again using Corollary  \ref{2}, we can conclude that $Ru \in  \mathcal{H}_{ \Lambda}^{m, 2}(\mathbb{Z}^n)$.  Using  the relation (\ref{113}),  we have $u\in  \mathcal{H}_{ \Lambda}^{m, 2}(\mathbb{Z}^n).$ 	Therefore 	
		$	Dom( T_{\sigma, 1}) \subseteq \mathcal{H}_{\Lambda}^{m}(\mathbb{Z}^n)$ and completes the proof.
		
	\end{proof}

	\section{Fredholmness and ellipticity}\label{sec6}
	This section is devoted to   study the Fredholmness and ellipticity of a pseudo-differential operators. We show that   a pseudo-differential operator of order 0 is $M$-elliptic if and only if it is a Fredholm operator. Finally we end this paper by    calculating  the index of such  pseudo-differential operators. We begin with the definition of Fredholm operators.
	\begin{definition}
		Let  $X$ and  $Y$ be two complex Banach spaces. Then a closed linear operator $T: X\to Y$   with dense domain $\mathcal{D}(T)$ is said to be Fredholm if
		\begin{enumerate}
			\item the range
			$R(T)$ of $T$ is a closed subspace of $Y$,
			\item the null space $N(T)$ of $T$   and the null space $N (T^{t} )$
			of the true adjoint $T^{t}$ of $T$ are finite dimensional.
		\end{enumerate} 
		For a Fredholm operator $T$, the index
		$i(T)$ of $T$ is defined by
		$$
		i(T)={\dim} ~N(T)-{\dim} ~N (T^{t} ).
		$$
	\end{definition}
	Another  criterion for Fredholm of a closed linear operator  is  given in the following thoreom due  to  Atkinson \cite{at}.
	\begin{theorem}\label{eq5000}
		Let $T$ be a closed linear operator from a complex Banach space $X$ into a
		complex Banach space $Y$ with dense domain $\mathcal{D}(T) .$ Then $T$ is Fredholm if and only if we can find a bounded linear operator $B: Y \rightarrow X$ and  compact operators $K_{1}: X \rightarrow X, K_{2}: Y \rightarrow Y$ such that $B T=I+K_{1}$ on $\mathcal{D}(T)$ and $T B=I+K_{2}$ on $Y$.
	\end{theorem}
	
	The main result of this section is the following theorem.
	\begin{theorem}
		Let $\sigma \in M_{\rho, \Lambda }^{0}\left(\mathbb{Z}^{n} \times \mathbb{T}^{n}\right) .$ Then $T_{\sigma}: \ell^{2}\left(\mathbb{Z}^{n}\right) \rightarrow \ell^{2}\left(\mathbb{Z}^{n}\right)$ is Fredholm if and only
		if $T_{\sigma}: \ell^{2}\left(\mathbb{Z}^{n}\right) \rightarrow \ell^{2}\left(Z^{n}\right)$ is $M$-elliptic.
	\end{theorem}
	\begin{proof}
		Suppose that $T_{\sigma}: \ell^{2}\left(\mathbb{Z}^{n}\right) \rightarrow \ell^{2}\left(\mathbb{Z}^{n}\right)$ is a Fredholm operator. Then $T_{\tau}=\mathcal{F}_{\mathbb{Z}^{n}} T_{\sigma}^{t} \mathcal{F}_{\mathbb{Z}^{n}}^{-1}$
		is also a Fredholm operator from $L^{2}\left(\mathbb{T}^{n}\right)$ into $L^{2}\left(\mathbb{T}^{n}\right)$, where
		$$
		\tau(x, k)=\overline{\sigma(-k, x)}, \quad x \in \mathbb{T}^{n}, k \in \mathbb{Z}^{n}
		$$
		Since $\sigma \in M_{\rho, \Lambda }^{0}\left(\mathbb{Z}^{n} \times \mathbb{T}^{n}\right) $, then    $\tau \in M_{\rho, \Lambda }^{0}\left(\mathbb{T}^{n} \times \mathbb{Z}^{n}\right)$. Then  by  Theorem 3.9 of \cite{shala},  $\tau$ is $M$-elliptic which implies  that $\sigma$ is $M$-elliptic. 
		
		Conversely, suppose that $\sigma$ is $M$-elliptic. Then by Theorem \ref{5}, there exists a symbol $\tau \in  M_{\rho, \Lambda }^{0}\left(\mathbb{Z}^{n} \times \mathbb{T}^{n}\right) $ such that
		$$
		T_{\sigma} T_{\tau}=I+R
		$$
		and
		$$
		T_{\tau} T_{\sigma}=I+S,
		$$
		where $R$ and $S$ are pseudo-differential operators with symbol in $\bigcap_{m \in \mathbb{R}}  M_{\rho, \Lambda }^{m}\left(\mathbb{Z}^{n} \times \mathbb{T}^{n}\right) .$ The pseudo-differential operator $R: \ell^{2}\left(\mathbb{Z}^{n}\right) \rightarrow \ell^{2}\left(\mathbb{Z}^{n}\right)$ is the same as
		the composition of the psudo-differential operator $R: \ell^{2}\left(\mathbb{Z}^{n}\right) \rightarrow 	\mathcal{H}_{ \Lambda}^{t, 2}(\mathbb{Z}^n)$ and the inclusion
		$i: 	\mathcal{H}_{ \Lambda}^{t, 2}(\mathbb{Z}^n)\rightarrow \ell^{2}\left(\mathbb{Z}^{n}\right)$. Since $R: \ell^{2}\left(\mathbb{Z}^{n}\right) \rightarrow 	\mathcal{H}_{ \Lambda}^{t, 2}(\mathbb{Z}^n)$ is a bounded linear operator and from
		Theorem \ref{eq17},  $i: 	\mathcal{H}_{ \Lambda}^{t, 2}(\mathbb{Z}^n) \rightarrow \ell^{2}\left(\mathbb{Z}^{n}\right)$ is compact, it follows that $R: \ell^{2}\left(\mathbb{Z}^{n}\right) \rightarrow \ell^{2}\left(\mathbb{Z}^{n}\right)$ is
		compact.  Similarly, we have  the pseudo-differential operator  $S: \ell^{2}\left(\mathbb{Z}^{n}\right) \rightarrow \ell^{2}\left(\mathbb{Z}^{n}\right)$ is compact. Hence by Theorem \ref{eq5000}, the operator $T_{\sigma}$ is
		Fredholm.
	\end{proof}
	Next we find   an index formula for the Fredholm pseudo-differential operators on $\mathbb{Z}^{n}$. From the Definition \ref{definition}, the relation (\ref{eq8}) and Lemma 4.1 of \cite{vish}, we have the following result.
	\begin{lemma}\label{eq11}
		Let $\mathcal{S}\left(\mathbb{Z}^{n} \times \mathbb{T}^{n}\right)$ be the Schwartz space on $\mathbb{Z}^{n} \times \mathbb{T}^{n}$. Then
		$$\bigcap_{m \in \mathbb{R}}  M_{\rho, \Lambda }^{m}\left(\mathbb{Z}^{n} \times \mathbb{T}^{n}\right)
		=\bigcap_{m \in \mathbb{R}}  S_{\rho, \Lambda }^{m}\left(\mathbb{Z}^{n} \times \mathbb{T}^{n}\right)=\mathcal{S}\left(\mathbb{Z}^{n} \times \mathbb{T}^{n}\right).
		$$
	\end{lemma}
	\begin{proof}
		Let $\sigma \in\bigcap_{m \in \mathbb{R}}  S_{\rho, \Lambda }^{m}\left(\mathbb{Z}^{n} \times \mathbb{T}^{n}\right).$   Let $\alpha, \beta$, $\tilde{\alpha}$ and $\tilde{\beta}$ be multi-indices in  $\mathbb{Z}^n$. Then there exists a real number $\nu_{1}$ such that
		$$
		\rho	|\alpha|+\nu_{1}-|\tilde{\beta}| \leq 0.
		$$
		Since $\sigma \in\bigcap_{m \in \mathbb{R}}  S_{\rho, \Lambda }^{m}\left(\mathbb{Z}^{n} \times \mathbb{T}^{n}\right)$ then $\sigma \in  S_{\rho, \Lambda }^{\nu_1}\left(\mathbb{Z}^{n} \times \mathbb{T}^{n}\right)$. Now 
		$$
		\begin{aligned}
			\sup _{(k, x) \in \mathbb{Z}^{n} \times \mathbb{T}^{n}}\left|k^{\alpha} x^{\beta} (D_{x}^{(\tilde{\alpha})} \Delta_{k}^{\tilde{\beta}} \sigma )(k, x)\right| & \leq \sup _{(k, x) \in \mathbb{Z}^{n} \times \mathbb{T}^{n}}|k|^{|\alpha|}\left| D_{x}^{(\tilde{\alpha})} \Delta_{k}^{\tilde{\beta}} \sigma(k, x)\right| \\
			& \leq C_{\nu_{1}, \tilde{\beta}, \tilde{\alpha}} \sup _{k \in \mathbb{Z}^{n}, x \in \mathbb{T}^{n}}\Lambda(k)^{\nu_1+\rho|\alpha|-|\tilde{\beta}|}<\infty,
		\end{aligned}
		$$
		where  $C_{\nu_{1}, \tilde{\beta}, \tilde{\alpha}}$ is a positive constant. This implies that  $\sigma \in \mathcal{S}\left(\mathbb{Z}^{n} \times \mathbb{T}^{n}\right)$. Other part  $\mathcal{S}\left(\mathbb{Z}^{n} \times \mathbb{T}^{n}\right) \subseteq \bigcap_{m \in \mathbb{R}}  S_{\rho, \Lambda }^{m}\left(\mathbb{Z}^{n} \times \mathbb{T}^{n}\right)$ follows directly. 
	\end{proof}

	Let $\sigma \in M_{\rho, \Lambda }^{0}\left(\mathbb{Z}^{n} \times \mathbb{T}^{n}\right)$ be $M$-elliptic.  Then there exists a symbol $\tau \in   M_{\rho, \Lambda }^{0}\left(\mathbb{Z}^{n} \times \mathbb{T}^{n}\right) $ such that 
	$$
	T_{\tau} T_{\sigma}=I-T_{1}
	$$
	and
	$$
	T_{\sigma} T_{\tau}=I-T_{2}
	$$
	where $T_{1}$ and $T_{2}$ are pseudo-differential operators with symbols $\tau_{1}$ and $\tau_{2}$   in $\bigcap_{m \in \mathbb{R}}  M_{\rho, \Lambda }^{m}\left(\mathbb{Z}^{n} \times \mathbb{T}^{n}\right)$. Let $f \in \mathcal{S}\left(\mathbb{Z}^{n} \times \mathbb{T}^{n}\right)$. Then for for all $k \in \mathbb{Z}^{n}$
	$$
	\left(T_{j} f\right)(k)=\int_{\mathbb{T}^{n}} e^{2 \pi i k \cdot x} \tau_{j}(k, x) \widehat{f}(x) d x, \quad j=1,2
	$$
	and by Lemma \ref{eq11}, for $j=1,2, \tau_{j} \in \mathcal{S}\left(\mathbb{Z}^{n} \times \mathbb{T}^{n}\right)$. Then $T_{1}$ and $T_{2}$ are trace class operators by Theorem 5.3 in \cite{bot}. Moreover, we have 
	$$
	\operatorname{Tr}\left(T_{j}\right)=\sum_{k \in \mathbb{Z}^{n}} \int_{\mathbb{T}^{n}} \tau_{j}(k, x) d x, \quad j=1, 2.
	$$
	Thus from  Theorem 20.13 of \cite{wong99}, the index 
	$i(T_\sigma)$  of $T_\sigma$ is given by   
	$$
	i\left(T_{\sigma}\right)=\operatorname{Tr}\left(T_{1}\right)-\operatorname{Tr}\left(T_{2}\right)=\sum_{k \in \mathbb{Z}^{n}} \int_{\mathbb{T}^{n}}\tau_{1}(k, x) -\tau_{2}(k, x)\; d x.
	$$

	\section*{Acknowledgement}
	
	VK is supported  by the FWO  Odysseus  1  grant  G.0H94.18N:  Analysis  and  Partial Differential Equations and  Partial Differential Equations and by the Methusalem programme of the Ghent University Special Research Fund (BOF) 	(Grant number 01M01021). SSM thanks IIT Guwahati, India, for the support  provided during the period of this work.

\end{document}